\documentclass[12pt]{amsart}
\usepackage[toc,page]{appendix}
\usepackage{amssymb, latexsym, euscript}
\usepackage{amssymb}
\usepackage{latexsym}
\usepackage{amsmath}

\def\sD{{\mathfrak D}}      
   \def\sH{{\mathfrak H}}   
   \def\sK{{\mathfrak K}}   \def\sL{{\mathfrak L}}
\def\sM{{\mathfrak M}}   \def\sN{{\mathfrak N}}

      \def\dC{{\mathbb C}}
\def\dD{{\mathbb D}}      
      
\def\dJ{{\mathbb J}}      
   \def\dN{{\mathbb N}}   
      \def\dR{{\mathbb R}}
   \def\dT{{\mathbb T}}   
   \def\dW{{\mathbb W}}   \def\dX{{\mathbb X}}

   \def\cB{{\mathcal B}}   
      
   \def\cH{{\mathcal H}}   
   \def\cK{{\mathcal K}}   \def\cL{{\mathcal L}}
      
      \def\cR{{\mathcal R}}
\def\cS{{\mathcal S}}   \def\cT{{\mathcal T}}   \def\cU{{\mathcal U}}

   \def\bB{{\mathbf B}}

      \def\bL{{\mathbf L}}
   \def\bN{{\mathbf N}}   
      
   \def\bT{{\mathbf T}}

\def\cRS{\mathcal{RS}}

\def\ran{{\rm ran\,}}
\def\cran{{\rm \overline{ran}\,}}
\def\dom{{\rm dom\,}}

\def\cdom{{\rm \overline{dom}\,}}
\def\clos{{\rm clos\,}}

\def\cspan{{\rm \overline{span}\, }}

\def\cmr{{\dC \backslash \dR}}
\def\uphar{{\upharpoonright\,}}

\def\f{\varphi}

\newtheorem{theorem}{Theorem}[section]
\newtheorem{lemma}[theorem]{Lemma}
\newtheorem{proposition}[theorem]{Proposition}
\newtheorem{corollary}[theorem]{Corollary}
\newtheorem{definition}[theorem]{Definition}
\newtheorem{remark}[theorem]{Remark}

\numberwithin{equation}{section}

\def\RE{{\rm Re\,}}
\def\IM{{\rm Im\,}}

\def\wt{\widetilde}
\def\wh{\widehat}

\def\f{\varphi}

\def\uphar{{\upharpoonright\,}}

\numberwithin{equation}{section}
\begin{document}
\title[Holomorphic operator valued functions]
{Holomorphic operator valued functions generated by passive
selfadjoint systems}
\author[Yury Arlinski\u{\i}]{Yu.M. Arlinski\u{\i}}
\address{
Department of Mathematics \\
Dragomanov National Pedagogical University \\
Kiev \\ Pirogova 9 \\ 01601 \\ Ukraine}
\email{yury.arlinskii@gmail.com}

\author[Seppo Hassi]{S. Hassi}
\address{Department of Mathematics and Statistics \\
University of Vaasa \\
P.O. Box 700 \\
65101 Vaasa \\
Finland} \email{sha@uwasa.fi} \dedicatory{Dedicated to Professor
Joseph Ball on the occasion of his 70-th birthday}
\thanks{This research was partially supported by a grant from the Vilho,
Yrj\"o and Kalle V\"ais\"al\"a Foundation of the Finnish Academy of
Science and Letters. Yu.M.~Arlinski\u{\i} also gratefully acknowledges
financial support from the University of Vaasa.}

\vskip 1truecm
\thispagestyle{empty}
\baselineskip=12pt

\date{\today}

\begin{abstract}
 Let $\mathfrak M$ be a Hilbert space. In this paper we study a class
$\mathcal R\mathcal S(\mathfrak M)$ of operator functions that are holomorphic in the domain
$\mathbb C\setminus\{(-\infty,-1]\cup [1,+\infty)\}$ and whose values are
bounded linear operators in $\sM$. The functions in $\mathcal R\mathcal S(\mathfrak M)$ are
Schur functions in the open unit disk $\mathbb D$ and, in addition,
Nevanlinna functions in $\mathbb C_+\cup\mathbb C_-$. Such functions can be
realized as transfer functions of minimal passive selfadjoint
discrete-time systems. We give various characterizations for the
class $\mathcal R\mathcal S(\mathfrak M)$ and obtain an explicit form for the inner
functions from the class $\mathcal R\mathcal S(\mathfrak M)$ as well as an inner dilation for
any function from $\mathcal R\mathcal S(\mathfrak M)$. We also consider various
transformations of the class $\mathcal R\mathcal S(\mathfrak M)$, construct realizations of
their images, and find corresponding fixed points.
\end{abstract}
\maketitle


\section{Introduction}

Throughout this paper we consider separable Hilbert spaces over the
field $\dC$ of complex numbers and certain classes of operator
valued functions which are holomorphic on the open upper/lower
half-planes $\dC_{+}/\dC_-$ and/or on the open unit disk $\dD$. A
$\bB(\sM)$-valued function $M$ is called a \textit{Nevanlinna
function} if it is holomorphic outside the real axis, symmetric
$M(\lambda)^*=M(\bar\lambda)$, and satisfies the inequality $\IM
\lambda\, \IM M(\lambda)\ge 0$ for all $\lambda\in\cmr$. This last
condition is equivalent to the nonnegativity of the kernel
\[
\cfrac{M(\lambda)-M(\mu)^*}{\lambda-\bar \mu}, \quad \lambda,\mu \in
\dC_+\cup \dC_-.
\]
On the other hand, a $\bB(\sM)$-valued function $\Theta(z)$ belongs
to the \textit{Schur class} if it is holomorphic on the unit disk
$\dD$ and contractive, $||\Theta(z)||\le 1$ $\forall z\in\dD$ or,
equivalently, the kernel
\[
 \frac{I-\Theta^*(w)\Theta(z)}{1-z\bar w},\quad z,w\in\dD
\]
is nonnegative. Functions from the Schur class appear naturally in
the study of linear discrete-time systems; we briefly recall some
basic terminology here; cf. D.Z.~Arov \cite{A,Arov}. Let $T$ be a
bounded operator given in the block form
\begin{equation}\label{abcd0}
T=\begin{bmatrix} D&C\cr B&A\end{bmatrix}:
\begin{array}{l}\sM\\\oplus\\\cK\end{array}\to
\begin{array}{l}\sN\\\oplus\\\cK\end{array}
\end{equation}
with separable Hilbert spaces $\sM,\sN$, and $\sK$. The system of
equations
\begin{equation}
\label{passive} \left\{
\begin{array}{l}
 h_{k+1}=Ah_k+B\xi_k,\\
 \sigma_k=Ch_k+D\xi_k,
\end{array}
\right. \qquad k\ge 0,
\end{equation}
describes the evolution of a \textit{linear discrete time-invariant
system} $\tau=\left\{T,\sM,\sN,\sK\right\}$. Here $\sM$ and $\sN$
are called the input and the output spaces, respectively, and $\sK$
is the state space. The operators $A$, $B$, $C$, and $D$ are called
the main operator, the control operator, the observation operator,
and the feedthrough operator of $\tau$, respectively. The subspaces
\begin{equation}
\label{CO} \sK^c =\cspan\{\,A^{n}B\sM:\,n \in \dN_0 \}
\quad\mbox{and}\quad \sK^o =\cspan\{\,A^{*n}C^*\sN:\,n \in \dN_0 \}
\end{equation}
are called the controllable and observable subspaces of
$\tau=\left\{T,\sM,\sN,\sK\right\}$, respectively. If $\sK^c=\sK$
($\sK^o=\sK$) then the system $\tau$ is said to be
\textit{controllable} (\textit{observable}), and \textit{minimal} if
$\tau$ is both controllable and observable. If $\sK=\clos
\{\sK^c+\sK^o\}$ then the system $\tau$ is said to be a
\textit{simple}. Closely related to these definitions is the notion
of $\sM$-simplicity: given a nontrivial subspace $\sM\subset\sH$
the operator $T$ acting in $\sH$ is said to be \textit{$\sM$-simple} if
\[
\overline{{\rm span\,}}\left\{\,T^n \sM,\,
 n\in\dN_0\right\}=\sH.
\]
Two discrete-time systems
$\tau_1=\left\{T_1,\sM,\sN,\sK_{1}\right\}$ and
$\tau_2=\left\{T_2,\sM,\sN,\sK_{2}\right\}$ are \textit{unitarily
similar} if there exists a unitary operator $U$ from $\sK_{1}$ onto
$\sK_{2}$ such that
\begin{equation}\label{unisim}
 A_2=UA_1U^*,\quad  B_2=UB_1, \quad  C_2=C_1U^*, \mbox{ and } D_2=D_1.
\end{equation}
If the linear operator $T$ is contractive (isometric, co-isometric,
unitary), then the corresponding discrete-time system is said to be
\textit{passive} (\textit{isometric, co-isometric, conservative}).
With the passive system $\tau$ in \eqref{passive} one associates the
\textit{transfer function} via
\begin{equation}
\label{TrFu} \Omega_\tau(z):=D+ z C(I - z A)^{-1}B, \quad z\in \dD.
\end{equation}
It is well known that the transfer function of a passive system
belongs to the \textit{Schur class} ${\bf S}(\sM,\sN)$ and,
conversely, that every operator valued function $\Theta(\lambda)$
from the Schur class ${\bf S}(\sM,\sN)$ can be realized as the
transfer function of a passive system, which can be chosen as
observable co-isometric (controllable isometric, simple
conservative, passive minimal). Notice that an application of the
Schur-Frobenius formula (see Appendix A) for the inverse of a block
operator gives with $\sM=\sN$ the relation
\begin{equation}\label{sch-fr}
P_\sM(I-zT)^{-1}\uphar\sM=(I_\sM-z\Omega_\tau(z))^{-1},\quad
z\in\dD.
\end{equation}

It is known that two isometric and controllable (co-isometric and
observable, simple conservative) systems with the same transfer
function are unitarily similar. However, D.Z.~Arov~\cite{A} has
shown that two minimal passive systems $\tau_1$ and $\tau_2$ with
the same transfer function $\Theta(\lambda)$ are only weakly
similar; weak similarity neither preserves the dynamical properties
of the system nor the spectral properties of its main operator $A$.
Some necessary and sufficient conditions for minimal passive systems
with the same transfer function to be (unitarily) similar have been
established in \cite{ArNu1,ArNu2}.

By introducing some further restrictions on the passive system
$\tau$ it is possible to preserve unitary similarity of passive
systems having the same transfer function. In particular, when the
main operator $A$ is normal such results have been obtained in
\cite{AHS3}; see in particular Theorem~3.1 and Corollaries 3.6--3.8
therein. A stronger condition on $\tau$ where main operator is
selfadjoint naturally yields to a class of systems which preserve
such a unitary similarity property. A class of such systems
appearing in \cite{AHS3} is the class of \textit{passive
quasi-selfadjoint systems}, in short \textit{$pqs$-systems}, which
is defined as follows: a collection
\[
\tau=\left\{ T,\sM,\sM,\cK\right\}
\]
is a $pqs$-system if the operator $T$ determined by the block
formula \eqref{abcd0} with the input-output space $\sM=\sN$ is a
contraction and, in addition,
$$\ran(T-T^*)\subseteq\sM.$$ Then, in particular, $F=F^*$ and
$B=C^*$ so that $T$ takes the form
\[
 T=\begin{bmatrix} D&C\cr
 C^*&F\end{bmatrix}:\begin{array}{l}\sM\\\oplus\\\cK\end{array}\to
\begin{array}{l}\sM\\\oplus\\\cK\end{array},
\]
i.e., $T$ is a quasi-selfadjoint contraction in the Hilbert space
$\sH=\sM\oplus\cK$. The class of $pqs$-systems gives rise to
transfer functions which belong to the subclass $\cS^{qs}(\sM)$ of
Schur functions. The class $\cS^{qs}(\sM)$ admits the following
intrinsic description; see \cite[Definition 4.4, Proposition
5.3]{AHS3}: a $\bB(\sM)$-valued function $\Omega$ belongs to
$\cS^{qs}(\sM)$ if it is holomorphic on
$\dC\setminus\{(-\infty,-1]\cup[1,+\infty)\}$ and has the following
additional properties:
\begin{itemize}
\item[(S1)]
$W(z)=\Omega(z)-\Omega(0)$ is a Nevanlinna function;
\item[(S2)]
the strong limit values $W(\pm 1)$ exist and $W(1)-W(-1)\le 2I$;
\item[(S3)]
$\Omega(0)$ belongs to the operator ball
$$\cB\left(-\frac{W(1)+W(-1)}{2},\; I-\frac{W(1)-W(-1)}{2}
\right)$$ with the center $-\cfrac{W(1)+W(-1)}{2}$ and with the left
and right radii $I-\cfrac{W(1)-W(-1)}{2}.$
\end{itemize}
It was proved in \cite[Theorem 5.1]{AHS3} that the class
$\cS^{qs}(\sM)$ coincides with the class of all transfer
functions of
$pqs$-systems with input-output space $\sM$. In
particular, every function from the class $\cS^{qs}(\sM)$ can be
realized as the transfer function of a \textit{minimal} $pqs$-system
and, moreover, two minimal realization are unitarily equivalent; see
\cite{AHS2,AHS3,ArlKlotz2010}. For $pqs$-systems the controllable
and observable subspaces $\cK^{c}$ and $\cK^{o}$ as defined in
\eqref{CO} necessarily coincide. Furthermore, the following
equivalences were established in \cite{ArlKlotz2010}:
\[
\begin{split}
  T \;\mbox{is}\;\; \sM\mbox{-simple }
 \Longleftrightarrow & \quad  \mbox{the operator}\; F \;\mbox{is}\;\cran
C^*-\mbox{simple in}\; \cK \\
 \Longleftrightarrow & \quad  \mbox{the system}\quad \tau=\left\{
\begin{bmatrix} D&C\cr
C^*&F\end{bmatrix},\sM,\sM,\cK\right\}\;\mbox{is minimal}.
\end{split}
\]

We can now introduce one of the main objects to be studied in the
present paper.

\begin{definition}\label{rsm}
Let $\sM$ be a Hilbert space. A $\bB(\sM)$-valued Nevanlinna
function $\Omega$ which is holomorphic on
$\dC\setminus\{(-\infty,-1]\cup[1,+\infty)\}$ is said to belong to
the class $\cRS(\sM)$ if
\[
 -I\le \Omega(x)\le I, \quad x\in (-1,1).
\]
The class $\cRS(\sM)$ will be called the combined Nevanlinna-Schur class of $\bB(\sM)$-valued operator functions.
\end{definition}

If $\Omega\in\cRS(\sM)$, then $\Omega(x)$ is non-decreasing on the
interval $(-1,1)$. Therefore, the strong limit values $\Omega(\pm
1)$ exist and satisfy the following inequalities
\begin{equation}\label{limits}
 -I_\sM\le \Omega(-1)\le \Omega(0)\le\Omega(1)\le I_\sM.
\end{equation}
It follows from (S1)--(S3) that the class $\cRS(\sM)$ is a subclass
of the class $\cS^{qs}(\sM)$.\\

In this paper we give some new characterizations of the class
$\cRS(\sM)$, find an explicit form for inner functions from the
class $\cR(\sM)$, and construct a bi-inner dilation for an arbitrary
function from $\cRS(\sM)$. For instance, in Theorem \ref{newchar} it
is proven that a $\bB(\sM)$-valued Nevanlinna function defined on
$\dC\setminus\{(-\infty,-1]\cup[1,+\infty)\}$ belongs to the class
$\cRS(\sM)$ if and only if
\[
 K(z,w):=I_\sM-\Omega^*(w)\Omega(z)-\cfrac{1-\bar w z}{z-\bar
 w}\,\left(\Omega(z)-\Omega^*(w)\right)
\]
defines a nonnegative kernel on the domains
\[
 \dC\setminus\{(-\infty,-1]\cup [1,+\infty)\},\quad \IM
z>0\quad\mbox{and}\quad  \dC\setminus\{(-\infty,-1]\cup
[1,+\infty)\},\quad \IM z<0.
\]

We also show that the transformation
\begin{equation}
\label{TREIN}
\cR\cS(\sM)\ni\Omega\mapsto{\bf\Phi}(\Omega)=\Omega_{\bf\Phi},\quad
 \Omega_{\bf\Phi}(z):=(zI-\Omega(z))(I-z\Omega(z))^{-1},
\end{equation}
with $z\in\dC\setminus\{(-\infty,-1]\cup[1,+\infty)\}$ is an
automorphism of $\cRS(\sM)$, ${\bf\Phi}^{-1}={\bf\Phi}$, and that
${\bf\Phi}$ has a unique fixed point, which will be specified in
Proposition \ref{fixpoint}.

It turns out that the set of inner functions from the class
$\cRS(\sM)$ can be seen as the image ${\bf \Phi}$ of constant
functions from $\cRS(\sM)$: in other words, the inner functions from
$\cRS(\sM)$ are of the form
\[
\Omega_{\rm {in}}(z)=(zI+A)(I+zA)^{-1},\;A\in [-I_\sM, I_\sM].
\]
In Theorem \ref{inerdil} it is proven that every function
$\Omega\in\cRS(\sM)$ admits the representation
\begin{equation}\label{biinerdil1}
 \Omega(z)=P_\sM\wt \Omega_{in}(z)\uphar\sM=P_\sM(zI+\wt A)(I+z\wt A)^{-1}\uphar\sM,\quad
 \wt A\in [-I_{\wt\sM}, I_{\wt\sM}],
\end{equation}
where $z\in\dC\setminus\{(-\infty,-1]\cup[1,+\infty)\}$ and $\wt\sM$
is a Hilbert space containing $\sM$ as a subspace and such that
$\cspan\{\wt A^n\sM:\;n\in\dN_0\}=\wt\sM$ (i.e., $\wt A$ is
$\sM$-simple). Equality \eqref{biinerdil1} means that an arbitrary
function of the class $\cRS(\sM)$ admits a bi-inner dilation (in the
sense of \cite{Arov}) that belongs to the class $\cRS(\wt\sM)$.

In Section \ref{sec6} we also consider the following transformations
of the class $\cRS(\sM)$:
\begin{multline}
\label{mappings2}
\Omega\left(\cfrac{z+a}{1+za}\right)=:\Omega_a
(z)\leftarrowtail\Omega(z)\rightarrowtail\wh\Omega_a(z):=(aI+\Omega(z))(I+a\Omega(z))^{-1},
\\
 a\in(-1,1),z\in\dC\setminus\{(-\infty,-1]\cup[1,+\infty)\}.
\end{multline}
These are analogs of the M\"{o}bius transformation
\[
w_a(z)=\cfrac{z+a}{1+az},\quad z\in\dC\setminus\{-a^{-1}\}
\;\left(a\in(-1,1),\;a\ne 0\right)
\]
of the complex plane. The mapping $w_a$ is an automorphism of
$\dC\setminus\{(-\infty,-1]\cup [1,+\infty)\}$ and it maps $\dD$
onto $\dD$, $[-1,1]$ onto $[-1,1]$, $\dT$ onto $\dT$, as well as
$\dC_+/\dC_-$ onto $\dC_+/\dC_-$.

The mapping
\[
\cRS(\sM)\ni\Omega
\mapsto \Omega_a(z)=\Omega\left(\cfrac{z+a}{1+za}\right)\in\cRS(\sM)
\]
can be rewritten as
\[
\Omega\mapsto\Omega\circ w_a.
\]
In Proposition \ref{nnnee} it is shown that the fixed points of this
transformation consist only of the constant functions from
$\cRS(\sM)$: $\Omega(z)\equiv A$ with $A\in [-I_\sM, I_\sM]$.

One of the operator analogs of $w_a$ is the following transformation
of $\bB(\sM)$:
\[
W_a(T)=(T+aI)(I+aT)^{-1},\quad a\in(-1,1).
\]
The inverse of $W_a$ is given by
\[
W_{-a}(T)=(T-aI)(I-aT)^{-1}.
\]
The class $\cRS(\sM)$ is stable under the transform $W_a$:
\[
 \Omega\in\cRS(\sM)\Longrightarrow W_a\circ\Omega\in\cRS(\sM).
\]
If $T$ is selfadjoint and unitary (a fundamental symmetry), i.e.,
$T=T^*=T^{-1}$, then for every $a\in (-1,1)$ one has
\begin{equation}\label{FunSym}
W_a(T)=T
\end{equation}
Conversely, if for a selfadjoint operator $T$ the equality
\eqref{FunSym} holds for some $a:-a^{-1}\in\rho(T)$, then $T$ is a
fundamental symmetry and \eqref{FunSym} is valid for all $a\ne\{\pm
1\}$.

One can interpret the mappings in \eqref{mappings2} as $\Omega\circ
w_a$ and $W_a\circ\Omega$, where $\Omega\in\cRS(\sM)$. Theorem
\ref{infix} states that inner functions from $\cRS(\sM)$ are the
only fixed points of the transformation
\[
\cR\cS(\sM)\ni\Omega\mapsto W_{-a}\circ\Omega\circ w_a.
\]
An equivalent statement is that the equality
$$\Omega\circ w_a=W_a\circ\Omega $$
holds only for inner functions $\Omega$ from the class $\cRS(\sM)$.
On the other hand, it is shown in Theorem \ref{cjdcyjd} that the only solutions of the functional equation
\[
 \Omega(z)=\left(\Omega\left(\cfrac{z-a}{1-az}\right)-a\,I_\sM\right)\left(I_\sM-a\,\Omega\left(\cfrac{z-a}{1-az}\right)\,\right)^{-1}
\]
in the class $\cRS(\sM)$, where $a\in(-1,1)$, $a\ne 0$, are constant functions $\Omega$, which are fundamental symmetries in $\sM$.

To introduce still one further transform, let
\[
 {\bf K}=\begin{bmatrix}
 K_{11}&K_{12}\cr K^*_{12}& K_{22}\end{bmatrix}:
 \begin{array}{l}\sM\\\oplus\\ H\end{array}\to \begin{array}{l}\sM\\\oplus\\ H\end{array}
\]
be a selfadjoint contraction and consider the mapping
\[
\cR\cS(H)\ni\Omega\;\mapsto\;\Omega_{\bf K} (z):=
K_{11}+K_{12}\Omega(z)(I-K_{22}\Omega(z))^{-1}K^*_{12},
\]
where $ z\in\dC\setminus\{(-\infty,-1]\cup[1,+\infty)\}.$ In Theorem
\ref{mappi} we prove that if $||K_{22}||<1,$ then $\Omega_{\bf
K}\in\cRS(\sM)$ and in Theorem \ref{redpro} we construct a
realization of $\Omega_{\bf K}$ by means of realization of
$\Omega\in\cRS(H)$ using the so-called \textit{Redheffer product};
see \cite{Redheffer,Timotin1995}. The mapping
\[
\bB(H)\ni T\mapsto K_{11}+K_{12}T(I-K_{22}T)^{-1}K_{21}\in \bB(\sM)
\]
can be considered as one further operator analog of the M\"{o}bius
transformation, cf. \cite{Shmul1}.

Finally, it is emphasized that in Section \ref{sec6} we will
systematically construct \textit{explicit realizations} for each of
the transforms ${\bf\Phi}(\Omega)$, $\Omega_a$, and $\wh \Omega_a$
as transfer functions of minimal passive selfadjoint systems using a
minimal realization of the initially given function
$\Omega\in\cRS(H)$.

{\bf Basic notations.} We use the symbols $\dom T$, $\ran T$, $\ker
T$ for the domain, the range, and the kernel of a linear operator
$T$. The closures of $\dom T$, $\ran T$ are denoted by $\cdom T$,
$\cran T$, respectively. The identity operator in a Hilbert space
$\sH$ is denoted by  $I$ and sometimes by $I_\sH$. If $\sL$ is a
subspace, i.e., a closed linear subset of $\sH$, the orthogonal
projection in $\sH$ onto $\sL$ is denoted by $P_\sL.$ The notation
$T\uphar \sL$ means the restriction of a linear operator $T$ on the
set $\sL\subset\dom T$. The resolvent set of $T$ is denoted by
$\rho(T)$. The linear space of bounded operators acting between
Hilbert spaces $\sH$ and $\sK$ is denoted by $\bB(\sH,\sK)$ and the
Banach algebra $\bB(\sH,\sH)$ by $\bB(\sH).$ For a contraction
$T\in\bB(\sH,\sK)$ the defect operator $(I-T^*T)^{1/2}$ is denoted
by $D_T$ and $\sD_T:=\cran D_T$. For defect operators one has the
commutation relations
\begin{equation}\label{commut}
 TD_T = D_{T^*}T, \quad T^*D_{T^*}=D_{T}T^*
\end{equation}
and, moreover,
\begin{equation}\label{ranTDT}
 \ran TD_T = \ran D_{T^*}T=\ran T \cap \ran D_{T^*}.
\end{equation}
In what follows we systematically use the Schur-Frobenius formula
for the resolvent of a block-operator matrix and parameterizations
of contractive block operators, see Appendices \ref{AppendixA} and
\ref{AppendixB}.

  \section{The combined Nevanlinna-Schur class $\mathcal R\mathcal S(\mathfrak M)$}
In this section some basic properties of operator functions belonging to the combined Nevanlinna-Schur class
$\cRS(\sM)$ are derived. As noted in Introduction every function
$\Omega\in\cRS(\sM)$ admits a realization as the transfer function
of a passive selfadjoint system. In particular, the function
$\Omega\uphar\dD$ belongs to the Schur class $\cS(\sM)$.

It is known from \cite{Arl1991} that, if $\Omega\in \cRS(\sM)$ then
for every $\beta\in[0,\pi/2)$ the following implications are
satisfied:
\begin{equation}\label{propert111}
\begin{array}{l}\left\{\begin{array}{l}
|z\sin\beta+i\cos\beta|\le 1\\
\quad z\ne\pm 1
\end{array}\right.
\Longrightarrow \left\|\Omega(z)\sin\beta+i\cos\beta\,I\right\|\le 1\\
\left\{\begin{array}{l} |z\sin\beta-i\cos\beta|\le 1\\
\quad z\ne\pm 1
\end{array}\right.
\Longrightarrow\left\|\Omega(z)\sin\beta-i\cos\beta\,I\right\|\le 1
\end{array}.
\end{equation}
In fact, in Section \ref{charact} these implications will be we
derived once again by means of some new characterizations for the
class $\cRS(\sM)$.

To describe some further properties of the class $\cRS(\sM)$
consider a passive selfadjoint system given by
\begin{equation}\label{Tsystem}
 \tau=\left\{\begin{bmatrix}
 D&C\cr C^*&F\end{bmatrix};\sM,\sM,\cK\right\},
\end{equation}
with $D=D^*$ and $F=F^*$. It is known, see
Proposition~\ref{ParContr} and Remark \ref{stut1} in Appendix
\ref{AppendixB}, that the entries of the selfadjoint contraction
\begin{equation}\label{Tblock}
T=\begin{bmatrix} D&C\cr
C^*&F\end{bmatrix}:\begin{array}{l}\sM\\\oplus\\
\cK\end{array}\to
\begin{array}{l}\sM\\\oplus\\
\cK\end{array}
\end{equation}
admit the parametrization
\begin{equation}\label{faul}
C=KD_F,\quad D=-KFK^*+ D_{K^*}Y D_{K^*},
\end{equation}
where $K\in\bB(\sD_F,\sM)$ is a contraction and $Y\in\bB(\sD_{K^*})$
is a selfadjoint contraction. The minimality of the system $\tau$
means that the following equivalent equalities hold:
\begin{equation}
\label{minim21}
\cspan\{F^nD_F K^*,\;n\in\dN_0\}=\cK\Longleftrightarrow\bigcap\limits_{n\in\dN_0}\ker(KF^n D_F)=\{0\}.
\end{equation}
Notice that if $\tau$ is minimal, then necessarily $\cK=\sD_F$ or,
equivalently, $\ker D_F=\{0\}$.

Recall from \cite{SF} the Sz.-Nagy -- Foias characteristic function
of the selfadjoint contraction $F$, which for every
$z\in\dC\setminus\{(-\infty,-1]\cup[1,+\infty)\}$ is given by
\[
\begin{split}
 \Delta_F(z) & =\left(-F+zD_F(I-zF)^{-1}D_F\right)\uphar\sD_F\\
             & =\left(-F+z(I-F^2)(I-zF)^{-1}\right)\uphar\sD_F\\
             & =(zI-F)(I-zF)^{-1}\uphar\sD_F.
\end{split}
\]
Using the above parametrization one obtains the representations, cf.
\cite[Theorem 5.1]{AHS3},
\begin{equation}\label{dshazom}
\begin{split}
\Omega(z)&=D+zC(I-zF)^{-1}C^*=D_{K^*}YD_{K^*}+K\Delta_F(z)K^*\\
 &= D_{K^*}YD_{K^*}+K(zI-F)(I-zF)^{-1}K^*.
\end{split}
\end{equation}
Moreover, this gives the following representation for the limit
values $\Omega(\pm 1)$:
\begin{equation}
\label{ghtltks} \Omega(-1)=-KK^*+D_{K^*}YD_{K^*},\quad
 \Omega(1)=KK^*+D_{K^*}YD_{K^*}.
\end{equation}
The case $\Omega(\pm1)^2=I_\sM$ is of special interest and can be
characterized as follows.

\begin{proposition}\label{neuvaa}
Let $\sM$ be a Hilbert space and let $\Omega\in\cRS(\sM)$. Then the
following statements are equivalent:
\begin{enumerate}
\def\labelenumi{\rm (\roman{enumi})}
\item $\Omega(1)^2=\Omega(-1)^2=I_\sM;$
\item  the equalities
\begin{equation}
\label{INNER1}
\begin{array}{l}
\left(\cfrac{\Omega(1)-\Omega(-1)}{2}\right)^2
=\cfrac{\Omega(1)-\Omega(-1)}{2},\\
\left(\cfrac{\Omega(1)+\Omega(-1)}{2}\right)^2
=I_\sM-\cfrac{\Omega(1)-\Omega(-1)}{2}
\end{array}
\end{equation}
hold;
\item if $\tau=\left\{T;\sM,\sM,\cK\right\}$ is a passive selfadjoint
system \eqref{Tsystem} with the transfer function $\Omega$ and if
the entries of the block operator $T$ are parameterized by
\eqref{faul}, then the operator $K\in\bB(\sD_F,\sM)$ is a partial
isometry and $Y^2=I_{\ker K^*}$.
\end{enumerate}

\end{proposition}
\begin{proof}
From \eqref{ghtltks} we get for all $f\in\sM$
\[
 ||f||^2-||\Omega(\pm 1)f||^2=||f||^2-||(D_{K^*}YD_{K^*}\pm KK^*)f||^2=||(K^*(I\mp
 Y)D_{K^*}f||^2+||D_YD_{K^*}f||^2;
\]
cf. \cite[Lemma~3.1]{AHS2007}. Hence
\begin{multline*}
\Omega(1)^2=\Omega(-1)^2=I_\sM\Longleftrightarrow
\left\{\begin{array}{l}K^*(I- Y)D_{K^*}=0\\
K^*(I+ Y)D_{K^*}=0\\
D_YD_{K^*}=0\end{array}\right.\Longleftrightarrow
\left\{\begin{array}{l}
K^*D_{K^*}=D_K K^*=0\\
K^*YD_{K^*}=0\\
D_YD_{K^*}=0\end{array}\right.\\
\Longleftrightarrow
\left\{\begin{array}{l}K \quad\mbox{is a partial isometry}\\
Y^2=I_{\sD_{K^*}}=I_{\ker K^*}
\end{array}\right..
\end{multline*}
Thus (i)$\Longleftrightarrow$(iii).

Since $K$ is a partial isometry, i.e., $KK^*$ is an orthogonal
projection, the formulas \eqref{ghtltks} imply that
$$K\quad\mbox{is a partial isometry}\Longleftrightarrow \left(\cfrac{\Omega(1)-\Omega(-1)}{2}\right)^2
=\cfrac{\Omega(1)-\Omega(-1)}{2},$$ and in this case $D_{K^*}Y=Y$,
which implies that
$$Y^2=I_{\sD_{K^*}}=I_{\ker K^*}\Longleftrightarrow\left(\cfrac{\Omega(1)+\Omega(-1)}{2}\right)^2
=I_\sM-\cfrac{\Omega(1)-\Omega(-1)}{2}.$$
Thus (iii) $\Longleftrightarrow$(ii).
\end{proof}

By interchanging the roles of the subspaces $\cK$ and $\sM$ as well
as the roles of the corresponding blocks of $T$ in \eqref{Tblock}
leads to the passive selfadjoint system
$$\eta=\left\{ \begin{bmatrix} D&C\cr C^*&F\end{bmatrix},\cK,\cK,\sM\right\}$$
now with the input-output space $\cK$ and the state space $\sM$. The
transfer function of $\eta$ is given by
\[
 B(z)=F+zC^*(I-zD)^{-1}C,\quad
 z\in\dC\setminus\{(-\infty,-1]\cup[1,+\infty)\}.
\]
By applying Appendix \ref{AppendixB} again one gets for \eqref{faul}
the following alternative expression to parameterize the blocks of
$T$:
\begin{equation}
\label{parame}
 C= D_D N^*,\quad  F=-NDN^*+D_{N^*}XD_{N^*},
\end{equation}
where $N:\sD_D\to\cK$ is a contraction and $X$ is a selfadjoint
contraction in $\sD_{N^*}$. Now, similar to \eqref{ghtltks} one gets
\[
 B(1)=NN^*+D_{N^*}XD_{N^*},\quad B(-1)=-NN^*+D_{N^*}XD_{N^*}.
\]
For later purposes, define the selfadjoint contraction $\wh F$ by
 \begin{equation}
 \label{jghtlta}
\wh F:=D_{N^*}XD_{N^*}=\cfrac{B(-1)+B(1)}{2}.
\end{equation}

The statement in the next lemma can be checked with a
straightforward calculation.

\begin{lemma}\label{LEM1}
Let the entries of the selfadjoint contraction
\[
T=\begin{bmatrix} D&C\cr
C^*&F\end{bmatrix}:\begin{array}{l}\sM\\\oplus\\
\cK\end{array}\to
\begin{array}{l}\sM\\\oplus\\
\cK\end{array}
\]
be parameterized by the formulas \eqref{parame} with a contraction
$N:\sD_D\to\cK$ and a selfadjoint contraction $X$ in $\sD_{N^*}$.
Then the function $W(\cdot)$ defined by
\begin{equation}\label{aeykw}
 W(z)=I+zDN^*\left(I-z\wh F\right)^{-1}N, \quad
 z\in\dC\setminus\{(-\infty,-1]\cup[1,+\infty)\},
\end{equation}
where $\wh F$ is given by \eqref{jghtlta}, is invertible and
\begin{equation}\label{aeykw1}
 W(z)^{-1}=I-zDN^*(I-zF)^{-1}N,\quad z\in\dC\setminus\{(-\infty,-1]\cup[1,+\infty)\}.
\end{equation}
\end{lemma}

The function $W(\cdot)$ is helpful for proving the next result.

\begin{proposition}\label{schaffen}
Let $\Omega\in\cRS(\sM)$. Then for all
$z\in\dC\setminus\{(-\infty,-1]\cup[1,+\infty)\}$ the function
$\Omega(z)$ can be represented in the form
\begin{equation}\label{mobinn}
\Omega(z)=\Omega(0)+D_{\Omega(0)}\Lambda(z)\left(I+\Omega(0)\Lambda(z)\right)^{-1}D_{\Omega(0)}
\end{equation}
with a function $\Lambda\in\cRS(\sD_{\Omega(0)})$ for which 
$\Lambda(z)=z\Gamma(z)$, where $\Gamma$ is a holomorphic $\bB(\sD_{\Omega(0)})$-valued function
such that $\|\Gamma(z)\|\le 1$ for $z\in\dD$. In particular, $\|\Lambda(z)\|\le|z|$ when
$z\in\dD$.
\end{proposition}

\begin{proof}
To prove the statement, let the function $\Omega$ be realized as the
transfer function of a passive selfadjoint system
$\tau=\left\{T;\sM,\sM,\cK\right\}$ as in \eqref{Tsystem}, i.e.
$\Omega(z)=D+zC(I-zF)^{-1}C^*$. Using \eqref{parame} rewrite
$\Omega$ as
\[
 \Omega(z)=D+zD_D N^*(I-zF)^{-1}N D_D
 =\Omega(0)+zD_{\Omega(0)}N^*(I-zF)^{-1}ND_{\Omega(0)}.
\]
The definition of $\wh F$ in \eqref{jghtlta} implies that the block
operator
\[
\begin{bmatrix} 0&N^* \cr
N&\wh F\end{bmatrix}:\begin{array}{l}\sD_{\Omega(0)}\\\oplus\\\cK\end{array}\to\begin{array}{l}\sD_{\Omega(0)}\\\oplus\\\cK\end{array}
\]
is a selfadjoint contraction (cf. Appendix \ref{AppendixB}).
Consequently, the $\bB(\sD_D)$-valued function
\begin{equation}\label{Lambda01}
 \Lambda(z):=zN^*\left(I_\cK-z\wh F\right)^{-1}N,\quad
z\in\dC\setminus\{(-\infty,-1]\cup[1,+\infty)\},
\end{equation}
is the transfer function of the passive selfadjoint system
\[
\tau_0=\left\{\begin{bmatrix} 0&N^* \cr
N&\wh F\end{bmatrix};\sD_{\Omega(0)},\sD_{\Omega(0)}, \cK\right\}
\]
Hence $\Lambda$ belongs the class $\cRS(\sD_{\Omega(0)})$.
Furthermore, using \eqref{aeykw} and \eqref{aeykw1} in
Lemma~\ref{LEM1} one obtains
\[
I+\Omega(0)\Lambda(z)=I+zDN^*\left(I-z\wh F\right)^{-1}N=W(z)
\]
and
\[
\left(I+\Omega(0)\Lambda(z)\right)^{-1}=W(z)^{-1}=I-zDN^*(I-zF)^{-1}N
\]
for all $z\in\dC\setminus\{(-\infty,-1]\cup[1,+\infty)\}.$ Besides,
in view of \eqref{parame} one has $\wh F-F=NDN^*$. This leads to the
following implications
\[
\begin{split}
 &N^*\left(I-\wh F\right)^{-1}N-N^*(I-zF)^{-1}N=zN^*\left(I-\wh F\right)^{-1}NDN^*(I-zF)^{-1}N\\
&\Longleftrightarrow zN^*\left(I-\wh F\right)^{-1}N\left(I-zDN^*(I-zF)^{-1}N\right)=zN^*(I-zF)^{-1}N \\
&\Longleftrightarrow \Lambda(z)\left(I+\Omega(0)\Lambda(z)\right)^{-1}=zN^*(I-zF)^{-1}N\\
&\Longrightarrow\Omega(z)=\Omega(0)+D_{\Omega(0)}\Lambda(z)\left(I+\Omega(0)\Lambda(z)\right)^{-1}D_{\Omega(0)}.
\end{split}
\]
Since $\Lambda(0)=0$, it follows from Schwartz's lemma that
$||\Lambda(z)||\le|z|$ for all $z$ with $|z|<1$. 
In particular, one has a factorization $\Lambda(z)=z\Gamma(z)$, where $\Gamma$ is a holomorphic $\bB(\sD_{\Omega(0)})$-valued function 
such that $\|\Gamma(z)\|\le 1$ for $z\in\dD$; this is also obvious from \eqref{Lambda01}.  
\end{proof}

One can verify that the following relation for $\Lambda(z)$ holds
\begin{equation}\label{mobinn2}
\Lambda(z)
=D^{(-1)}_{\Omega(0)}(\Omega(z)-\Omega(0))(I-\Omega(0)\Omega(z))^{-1}
D_{\Omega(0)},
\end{equation}
where $D^{(-1)}_{\Omega(0)}$ stands for the Moore-Penrose inverse of
$D_{\Omega(0)}$.

It should be noted that the formula \eqref{mobinn} holds for all
$z\in\dC\setminus\{(-\infty,-1]\cup[1,+\infty)\}$. A general Schur
class function $\Omega\in{\bf S}(\sM,\sN)$ can be represented in the
form
\[
 \Omega(z)=\Omega(0)+D_{\Omega(0)^*}\Lambda(z)\left(I+\Omega(0)^*\Lambda(z)\right)^{-1}D_{\Omega(0)},
 \quad z\in\dD.
\]
This is called a M\"obius representation of $\Omega$ and it can be
found in \cite{BL1976,Cons,Shmul1}.

\section{Inner functions from the class $\mathcal R\mathcal S(\mathfrak M)$}
An operator valued function from the Schur class is called
\textit{inner/co-inner} (or \textit{$*$-inner}) (see e.g. \cite{SF})
if it takes isometric/co-isometric values almost everywhere on the
unit circle $\dT$, and it is said to be  \textit{bi-inner} when it
is both inner and co-inner.

Observe that if $\Omega\in\cRS(\sM)$ then $\Omega(z)^*=\Omega(\bar
z)$. Since $\dT\setminus\{-1,1\}
\subset\dC\setminus\{(-\infty,-1]\cup[1,+\infty)\}$, one concludes
that $\Omega\in\cRS(\sM)$ is inner (or co-inner) precisely when it
is bi-inner. Notice also that every function $\Omega\in\cRS(\sM)$
can be realized as the transfer function of a minimal passive
selfadjoint system $\tau$ as in \eqref{Tsystem}; cf. \cite[Theorem
5.1]{AHS3}.

The next statement contains a characteristic result for transfer
functions of conservative selfadjoint systems.

\begin{proposition}
\label{selfaconserv} Assume that the selfadjoint system
$\tau=\left\{T;\sM,\sM,\cK\right\}$ in \eqref{Tsystem} is
conservative. Then its transfer function
$\Omega(z)=D+zC(I_\cK-zF)^{-1}C^*$ is bi-inner and it takes the form
\begin{equation}\label{forminner}
 \Omega(z)=(z I_\sM+D)(I_\sM+zD)^{-1},\quad
 z\in\dC\setminus\{(-\infty,-1]\cup[1,+\infty)\}.
\end{equation}
On the other hand, if $\tau$ is a minimal passive selfadjoint system
whose transfer function is inner, then $\tau$ is conservative.
\end{proposition}
\begin{proof}
Let the entries of $T$ in \eqref{Tblock} be parameterized as in
\eqref{parame}. By assumption $T$ is unitary and hence
$N\in\bB(\sD_D,\cK)$ is isometry and $X$ is selfadjoint and unitary
in the subspace $\sD_{N^*}=\ker N^*$; see Remark \ref{herbst} in
Appendix \ref{AppendixB}. Thus $NN^*$ and $D_{N^*}$ are orthogonal
projections and $NN^*+D_{N^*}=I_\cK$ which combined with
\eqref{parame} leads to
\[
\begin{split}
 \left(I_\cK-zF\right)^{-1}
  &=\left(N(I+z D)N^*+D_{N^*}(I-zX)D_{N^*}\right)^{-1}\\
  &=N(I+zD)^{-1}N^*+D_{N^*}(I-zX)^{-1}D_{N^*},
\end{split}
\]
and, consequently,
\[
\begin{split}
 \Omega(z)&=D+zC(I_\cK-zF)^{-1}C^*\\
 &=D+zD_DN^*\left(N(I+zD)^{-1}N^*+D_{N^*}(I-zX)^{-1}D_{N^*}\right)ND_D\\
 &=D+z(I+zD)^{-1}D^2_D=(z I_\sM+D)(I_\sM+zD)^{-1},
\end{split}
\]
for all $z\in\dC\setminus\{(-\infty,-1]\cup[1,+\infty)\}$. This
proves \eqref{forminner} and this clearly implies that $\Omega(z)$
is bi-inner.

To prove the second statement assume that the transfer function of a
minimal passive selfadjoint system $\tau$ is inner. Then it is
automatically bi-inner. Now, according to a general result of
D.Z.~Arov \cite[Theorem 1]{Arov} (see also \cite[Theorem 1]{ArNu2},
\cite[Theorem 1.1]{AHS2007}), if $\tau$ is a passive simple
discrete-time system with bi-inner transfer function, then $\tau$ is
conservative and minimal. This proves the second statement.
\end{proof}

The formula \eqref{forminner} in Proposition \ref{selfaconserv}
gives a one-to-one correspondence between the operators $D$ from the
operator interval $[-I_\sM, I_\sM]$ and the inner functions from the
class $\cRS(\sM)$. Recall that for $\Omega\in\cRS(\sM)$ the strong
limit values $\Omega(\pm 1)$ exist as selfadjoint contractions; see
\eqref{limits}. The formula \eqref{forminner} shows that if
$\Omega\in\cRS(\sM)$ is an inner function, then necessarily these
limit values are also unitary:
\begin{equation}\label{novinner}
\Omega(1)^2=\Omega(-1)^2=I_\sM.
\end{equation}
However, these two conditions do not imply that $\Omega\in\cRS(\sM)$
is an inner function; cf. Proposition \ref{neuvaa} and Remark
\ref{herbst} in Appendix \ref{AppendixB}.

The next two theorems offer some sufficient conditions for
$\Omega\in\cRS(\sM)$ to be an inner function. The first one shows
that by shifting $\xi\in\dT$ ($|\xi|=1$) away from the real line
then existence of a unitary limit value $\Omega(\xi)$ at a single
point implies that $\Omega\in\cRS(\sM)$ is actually a bi-inner
function.

\begin{theorem} \label{ljcndy}
Let $\Omega$ be a nonconstant function from the class $\cRS(\sM)$.
If $\Omega(\xi)$ is unitary for some $\xi_0\in\dT$, $\xi_0\ne\pm 1$.
Then $\Omega$ is a bi-inner function.
\end{theorem}

\begin{proof}
 Let $\tau=\left\{T;\sM,\sM,\cK\right\}$ in \eqref{Tsystem} be a minimal passive
selfadjoint system whose transfer function is $\Omega$ and let the
entries of $T$ be parameterized as in \eqref{faul}. Using the
representation \eqref{dshazom} one can derive the following formula
for all $\xi\in\dT\setminus\{\pm 1\}$:
\[
\left\|D_{\Omega (\xi)}h\right\|^2=
\|D_{\Delta_{F}(\xi)}K^*h\|^2+\|D_{Y}D_{K^*}h\|^2
+\left\|\left(D_{K}\Delta_{F}(\xi)K^*-K^*YD_{K^*}\right)h\right\|^2;
\]
cf. \cite[Theorem~5.1]{AHS2007}, \cite[Theorem~2.7]{AHS3}. Since
$\Delta_F(\xi)$ is unitary for all $\xi\in\dT\setminus\{\pm 1\}$ and
$\Omega(\xi_0)$ is unitary, one concludes that $Y$ is unitary on
$\sD_{K^*}$ and
$\left(D_{K}\Delta_{F}(\xi_0)K^*-K^*YD_{K^*}\right)h=0$ for all
$h\in\sM$.

Suppose that there is $h_0\ne 0$ such that
$D_{K}\Delta_{F}(\xi_0)K^*h_0\ne 0$ and $K^*YD_{K^*}h_0\ne 0$. Then,
due to $D_{K}\Delta_{F}(\xi_0)K^*h_0=K^*YD_{K^*}h_0$, the
equalities $D_K K^*=K^*D_{K^*}$, and
\[
 \ran D_K\cap \ran K^*=\ran D_K K^*=\ran K^*D_{K^*},
\]
see \eqref{commut}, \eqref{ranTDT}, one concludes that there exists
$\f_0\in\sD_{K^*}$ such that
\[
\left\{\begin{array}{l} \Delta_F(\xi_0)K^*h_0=K^*\f_0\\ YD_{K^*}h_0=D_{K^*}\f_0\end{array}\right..
\]
Furthermore, the equality $D_{\Omega (\xi_0)^*}=D_{\Omega
(\bar\xi_0)}=0$ implies
$\left(D_{K}\Delta_{F}(\bar\xi_0)K^*-K^*YD_{K^*}\right)h=0$ for all
$h\in\sM$. Now $YD_{K^*}h_0=D_{K^*}\f_0$ leads to
$\Delta_F(\bar\xi_0)K^*h_0=K^*\f_0$. It follows that
\[
\Delta_F(\xi_0)K^*h_0=\Delta_F(\bar\xi_0)K^*h_0.
\]
Because
$\Delta_F(\bar\xi_0)=\Delta_F(\xi_0)^*=\Delta_F(\xi_0)^{-1}$, one
obtains $\left(I-\Delta_F(\xi_0)^2\right)K^*h_0=0.$ From
\[
\Delta_F(\xi_0)=(\xi_0 I-F)(I-\xi_0F)^{-1}
\]
it follows that
\[
(1-\xi_0^2)(I-\xi_0 F)^{-2}(I-F^2)K^*h_0=0.
\]
Since $\ker D_F=\{0\}$ (because the system $\tau$ is minimal), we
get $K^*h_0=0$. Therefore, $D_{K}\Delta_{F}(\xi_0)K^*h_0=0$ and
$K^*YD_{K^*}h_0=0$. One concludes that
\[
\left\{\begin{array}{l} D_{K}\Delta_F(\xi_0)K^*h=0\\K^*YD_{K^*}h=0\end{array}\right. \;\forall h\in\sM.
\]
The equality $\ran Y=\sD_{K^*}$ implies $K^*D_{K^*}=D_K K^*=0$.
Therefore $K$ is a partial isometry. The equality
$D_{K}\Delta_F(\xi_0)K^*=0$ implies $\ran
(\Delta_F(\xi_0)K^*)\subseteq\ran K^*$. Representing
$\Delta_F(\xi_0)$ as
\[
\Delta_F(\xi_0)= (\xi_0 I-F)(I-\xi_0 F)^{-1}K^*=\left(\bar\xi_0 I+(\xi_0-\bar\xi_0)(I-\xi_0 F)^{-1}\right)K^*,
\]
we obtain that $F(\ran K^*)\subseteq\ran K^*$. Hence $F^nD_F(\ran
K^*)\subseteq\ran K^*$ for all $n\in\dN_0$. Because the system
$\tau$ is minimal it follows that $\ran K^*=\sD_F=\cK$, i.e., $K$ is
isometry and hence $T$ is unitary (see Appendix \ref{AppendixB}).
This implies that $D_{\Omega(\xi)}=0$ for all $\zeta\in
\dT\setminus\{-1,1\}$, i.e., $\Omega$ is inner and, thus also
bi-inner.
\end{proof}
\begin{theorem} \label{thinne}
Let $\Omega\in\cRS(\sM)$. If the equalities \eqref{novinner} hold
and, in addition, for some $a\in(-1,1)$, $a\ne 0$, the equality
\begin{equation}
\label{omoa}
(\Omega(a)-aI_\sM)(I_\sM-a\Omega(a))^{-1}=\Omega(0)
\end{equation}
is satisfied, then $\Omega$ is bi-inner.
\end{theorem}
\begin{proof}
Let $\tau=\left\{T;\sM,\sM,\cK\right\}$ be a minimal passive
selfadjoint system as in \eqref{Tsystem} with the transfer function
$\Omega$ and let the entries of $T$ in \eqref{Tblock} be
parameterized as in \eqref{faul}. According to Proposition
\ref{neuvaa} the equalities \eqref{novinner} mean that $K$ is a
partial isometry and $Y^2=I_{\ker K^*}$.

Since $D_{K^*}$ is the orthogonal projection, $\ran Y\subseteq\ran D_{N^*}$, from \eqref{dshazom} we have
$$\Omega(z)=YD_{K^*}+K(zI-F)(I-zF)^{-1}K^*.$$
Rewrite \eqref{omoa} in the form
\begin{equation}\label{zeqz}
\Omega(0)(I_\sM-a\Omega(a))=\Omega(a)-aI_\sM.
\end{equation}
This leads to
\begin{multline*}
(-KFK^*+ YD_{K^*})\left(I_\sM-a\left(YD_{K^*}+K(aI-F)(I-aF)^{-1}K^*\right)\right)\\
=YD_{K^*}+K(aI-F)(I-aF)^{-1}K^*-aI_\sM,
\end{multline*}

\begin{multline*}
(-KFK^*+ YD_{K^*})\left((I-a Y)D_{K^*}+K\bigl(I-a\left(aI-F\right)\left(I-aF\right)^{-1}\right)K^*\bigr)\\
=(Y-aI)D_{K^*}+K\left((aI-F)(I-aF)^{-1}-a I\right)K^*,
\end{multline*}

\begin{multline*}
-KFK^*K\left(I-a(aI-F)(I-aF)^{-1}\right)K^*+ Y(I-a Y)D_{K^*}\\
=(Y-aI)D_{K^*}+K\left((aI-F)(I-aF)^{-1}-a I\right)K^*.
\end{multline*}
Let $P$ be an orthogonal projection from $\cK$ onto $\ran K^*$.
Since $K$ is a partial isometry, one has $K^*K=P.$ The equality
$Y^2=I_{\sD_{K^*}}$ implies $Y(I-a Y)D_{K^*}=(Y-aI)D_{K^*}$. This
leads to the following identities:
\[
\begin{array}{l}
K\biggl(-FP\left(I-a(aI-F)(I-aF)^{-1}\right)-
(aI-F)(I-aF)^{-1}+a I\biggr)K^*=0,\\[3mm]
KF(I_\sM-P)(I-aF)^{-1}K^*=0,\\[3mm]
PF(I_\sM-P)(I-aF)^{-1}P=0.
\end{array}
\]
Represent the operator $F$ in the block form
\[
F=\begin{bmatrix} F_{11}&F_{12}\cr F_{12}^*& F_{22} \end{bmatrix}:\begin{array}{l}\ran P\\\oplus\\\ran(I-P)\end{array}\to
 \begin{array}{l}\ran P\\\oplus\\\ran(I-P)\end{array}.
\]
Define
\[
\Theta(z)=F_{11}+zF_{12}(I-zF_{22})^{-1}F_{12}^*.
\]
Since $F$ is a selfadjoint contraction, the function $\Theta$
belongs to the class $\cRS(\ran P)$. From the Schur-Frobenius
formula \eqref{Sh-Fr1} it follows that
\[
(I-P)(I-aF)^{-1}P=a(I-aF_{22})^{-1}F^*_{12}(I-a\Theta(a))^{-1}P.
\]
This equality yields the equivalences
\begin{multline*}
PF(I_\sM-P)(I-aF)^{-1}P=0\Longleftrightarrow F_{12}(I-aF_{22})^{-1}F^*_{12}(I-a\Theta(a))^{-1}P=0\\
\Longleftrightarrow F_{12}(I-aF_{22})^{-1}F^*_{12}=0 \Longleftrightarrow (I-aF_{22})^{-1/2}F^*_{12}=0 \Longleftrightarrow F^*_{12}=0.
\end{multline*}
It follows that the subspace $\ran K^*$ reduces $F$. Hence $\ran
K^*$ reduces $D_F$ and, therefore $F^nD_F\ran K^*\subseteq\ran K^*$
for an arbitrary $n\in\dN_0$. Since the system $\tau$ is minimal, we
get $\ran K^*=\cK$ and this implies that $K$ is an isometry. Taking
into account that $Y^2=I_{\sD_{K^*}}$, we get that the block
operator $T$ is unitary. By Proposition \ref{selfaconserv} $\Omega$
is bi-inner.
\end{proof}

For completeness we recall the following result on the limit values
$\Omega(\pm1)$ of functions $\Omega\in{\bf S}^{qs}(\sM)$ from
\cite[Theorem 5.8]{AHS3}.

\begin{lemma}\label{INNER}
Let $\sM$ be a Hilbert space and let $\Omega\in{\bf S}^{qs}(\sM)$.
Then:
\begin{enumerate}
\item
if $\Omega(\lambda)$ is inner then
\begin{equation}
\label{INNER1b}
\begin{array}{l}
 \left(\dfrac{\Omega(1)-\Omega(-1)}{2}\right)^2
  =\dfrac{\Omega(1)-\Omega(-1)}{2},\\[5mm]
 (\Omega(1)+\Omega(-1))^*(\Omega(1)+\Omega(-1))
  =4I_\sM-2\left(\Omega(1)-\Omega(-1)\right);
\end{array}
\end{equation}
\item
if $\Omega$ is co-inner then
\begin{equation}
\label{COINNER1}
\begin{split}
&\left(\frac{\Omega(1)-\Omega(-1)}{2}\right)^2
=\frac{\Omega(1)-\Omega(-1)}{2},\\[3mm]
&(\Omega(1)+\Omega(-1))(\Omega(1)+\Omega(-1))^*
=4I_\sM-2\left(\Omega(1)-\Omega(-1)\right);
\end{split}
\end{equation}
\item
if \eqref{INNER1b}/\eqref{COINNER1} holds and $\Omega(\xi)$ is
isometric/co-isometric  for some $\xi \in\dT$, $\xi\ne \pm 1$, then
$\Omega$ is inner/co-inner.
\end{enumerate}
\end{lemma}

\begin{proposition}\label{normalop}
If $\Omega\in\cRS(\sM)$ is an inner function, then
$$\Omega(z_1)\Omega(z_2)=\Omega(z_2)\Omega(z_1), \quad
\forall z_1,z_2\in\dC\setminus\{(-\infty,-1]\cup[1,+\infty)\}.$$
In particular, $\Omega(z)$ is a normal operator for each $z\in\dC\setminus\{(-\infty,-1]\cup[1,+\infty)\}$.
\end{proposition}
\begin{proof}
The commutativity property follows from \eqref{forminner}, where
$D=\Omega(0)$. Normality follows from commutativity and symmetry
$\Omega(z)^*=\Omega(\bar z)$ for all $z$.
\end{proof}

\section{Characterization of the class $\mathcal R\mathcal S(\mathfrak M)$}\label{charact}

\begin{theorem}
\label{newchar} Let $\Omega$ be an operator valued Nevanlinna
function defined on $\dC\setminus\{(-\infty,-1]\cup[1,+\infty)\}$.
Then the following statements are equivalent:
\begin{enumerate}
\def\labelenumi{\rm (\roman{enumi})}
\item $\Omega$
 belongs to the class $\cRS(\sM)$;
\item $\Omega$ satisfies the inequality
\begin{equation}
\label{ythfd1}
 I-\Omega^*(z)\Omega(z)-(1-|z|^2)\cfrac{\IM \Omega (z)}{\IM z}\ge
 0,\quad \IM z\ne 0;
\end{equation}

\item the function
\[
K(z,w):=I-\Omega^*(w)\Omega(z)-\cfrac{1-\bar w z}{z-\bar w}\,\left(\Omega(z)-\Omega^*(w)\right)
\]
is a nonnegative kernel on the domains
$$\dC\setminus\{(-\infty,-1]\cup [1,+\infty)\},\; \IM z>0\quad\mbox{and}\quad  \dC\setminus\{(-\infty,-1]\cup [1,+\infty)\},\; \IM z<0;$$

\item the function
\begin{equation}
\label{formula3}
\Upsilon(z)=\left(zI-\Omega(z)\right)\left(I-z\Omega(z)\right)^{-1},\quad
z\in\dC\setminus\{(-\infty,-1]\cup[1,+\infty)\},
\end{equation}
is well defined and belongs to $\cRS(\sM)$.

\end{enumerate}
\end{theorem}
\begin{proof} (i)$\Longrightarrow$(ii) and (i)$\Longrightarrow$(iii). Assume that $\Omega\in
\cRS(\sM)$ and let $\Omega$ be represented as the the transfer
function of a passive selfadjoint system
$\tau=\left\{T;\sM,\sM,\cK\right\}$ as in \eqref{Tsystem} with the
selfadjoint contraction $T$ as in \eqref{faul}. According to
\eqref{dshazom} we have
\[\Omega(z)= D_{K^*}YD_{K^*}+K\Delta_F(z)K^*,\; z\in \dC\setminus\{(-\infty,-1]\cup[1,+\infty)\}.
\]
Taking into account that, see \cite[Chapter VI]{SF},
\[
((I-\Delta^*_F(w)\Delta_F(z))\f,\psi)=(1-\bar w z)((I-z F)^{-1}D_F\f, (I-w F)^{-1}D_F\psi)%
\]
and
\[
((\Delta_F(z)-\Delta^*_F(w))\f,\psi)=(z-\bar w)((I-zF)^{-1}D_F\f,(I-wF)^{-1}D_F\psi),
\]
we obtain
\[
\begin{array}{ll}
 ||h||^2-||\Omega(z) h||^2 & =||K^* h||^2-||\Delta_F(z) K^*h||^2\\
 &\quad
 +||D_Y D_{K^*}h||^2+||(K^*YD_{K^*}-D_{K}\Delta_F(z)K^*)h||^2 \\[3mm]
 &=(1-|z|^2)||(I-zF)^{-1}D_F K^*h||^2+||D_Y D_{K^*}h||^2\\
 &\quad +||(K^*YD_{K^*}-D_{K}\Delta_F(z)K^*)h||^2.
\end{array}
\]
Moreover,
\[
 \IM(\Omega(z)h,h)=\IM z||(I-zF)^{-1}D_FK^*h||^2
\]
and
\begin{multline*}
\IM z(||h||^2-||\Omega(z) h||^2)-(1-|z|^2)\IM (\Omega (z)h,h)\\
=\IM z\left(||D_Y
D_{K^*}h||^2+||(K^*YD_{K^*}-D_{K}\Delta_F(z)K^*)h||^2\right).
\end{multline*}
Similarly,
\begin{multline}\label{zlhjk}
(K(z,w)f,g)=((I-\Omega^*(w)\Omega(z))f,g)-\cfrac{1-\bar w z}{z-\bar w}\,((\Omega(z)-\Omega^*(w))f,g)\\
=(D^2_{Y}D_{K^*}f,D_{K^*}g)+\left((D_K\Delta_F(z)K^*-K^*YD_{K^*})f,(D_K\Delta_F(w)K^*-K^*YD_{K^*})g\right).
\end{multline}
It follows from \eqref{zlhjk} that for arbitrary complex numbers
$\{z_k\}_{k=1}^m\subset\dC\setminus\{(-\infty,-1]\cup
[1,+\infty)\}$, $\IM z_k>0$, $k=1,\dots,n$ or
$\{z_k\}_{k=1}^m\subset\dC\setminus\{(-\infty,-1]\cup
[1,+\infty)\}$, $\IM z_k<0$, $k=1,\dots,n$ and for arbitrary vectors
$\{f_k\}_{k=1}^\infty\subset\sM$ the relation
\[
\sum\limits_{k=1}^n(K(z_k,z_m)f_k,f_m)=\left\|D_{Y}D_{K^*}\sum_{k=1}^\infty f_k\right\|^2+
\left\|\sum_{k=1}^\infty (D_{K}\Delta_F(z_k)K^*-K^*YD_{K^*})f_k\right\|^2
\]
holds.
Therefore $K(z,w)$ is a nonnegative kernel.

(iii)$\Longrightarrow$(ii) is evident.

(ii)$\Longrightarrow$(iv) Because $\IM z >0$ ($\IM z<0$)
$\Longrightarrow $ $\IM\Omega(z)\ge 0$ ($\IM \Omega (z)\le 0)$, the
inclusion $1/z\in\rho(\Omega(z))$ is valid for $z$ with $\IM z\ne
0$. In addition $1/x\in\rho(\Omega(x))$ for $x\in(-1,1),$ $x\ne 0$,
because $\Omega(x)$ is a contraction. Hence $\Upsilon(z)$ is well
defined on $\sM$  and $\Upsilon^*(z)=\Upsilon(\bar z)$ for all
$z\in\dC\setminus\{(-\infty,-1]\cup[1,+\infty)\}.$ Furthermore, with
$\IM z\ne 0$ one has
\[
 \IM\Upsilon(z)
 =\left(I-\bar z\Omega^*(z)\right)^{-1}\left[\IM z(I-\Omega^*(z)\Omega(z))-(1-|z|^2)\IM\Omega(z)\right]\left(I- z\Omega(z)\right)^{-1},
\]
while for $x\in(-1,1)$
\[
  I-\Upsilon^2(x)=(1-x^2)\left(I-x\Omega(x)\right)^{-1}(I-\Omega^2(x))\left(I-x\Omega(x)\right)^{-1}.
\]
Thus, $\Upsilon\in\cRS(\sM)$.

(iv)$\Longrightarrow$(i) It is easy to check that if $\Upsilon$ is
given by \eqref{formula3}, then
\[
\Omega(z)=\left(zI-\Upsilon(z)\right)\left(I-z\Upsilon(z)\right)^{-1},\;
z\in\dC\setminus\{(-\infty,-1]\cup[1,+\infty)\}.
\]
Hence, this implication reduces back to the proven implication
(i)$\Longrightarrow$(ii).
\end{proof}

\begin{remark}
\label{einmal}
1) Inequality \eqref{ythfd1} can be rewritten as follows
\[
\left((I-\Omega^*(z)\Omega(z))f,f\right)-\cfrac{1-|z|^2}{|\IM z|}\left|{\IM (\Omega (z)f,f)}\right|\ge
 0,\quad \IM z\ne 0,\; f\in\sM.
\]
Let $\beta\in [0,\pi/2]$. Taking into account that
$$|z\sin\beta\pm i\cos\beta|^2=1\Longleftrightarrow 1-|z|^2=\pm 2\cot\beta\,\IM z$$
one obtains, see \eqref {propert111},
\[
\begin{array}{l}\left\{\begin{array}{l}
|z\sin\beta+i\cos\beta|= 1\\
\quad z\ne\pm 1
\end{array}\right.
\Longrightarrow \left\|\Omega(z)\sin\beta+i\cos\beta\,I\right\|\le 1\\
\left\{\begin{array}{l} |z\sin\beta-i\cos\beta|=1\\
\quad z\ne\pm 1
\end{array}\right.
\Longrightarrow\left\|\Omega(z)\sin\beta-i\cos\beta\,I\right\|\le 1
\end{array}.
\]
2) Inequality \eqref{ythfd1} implies
\[
 I-\Omega^*(x)\Omega(x)-(1-x^2)\Omega'(x)\ge 0, \quad x\in (-1,1).
\]
3) Formula \eqref{forminner} implies that if $\Omega\in\cRS(\sM)$ is
an inner function, then
\[
I-\Omega^*(w)\Omega(z)-\cfrac{1-\bar w z}{z-\bar w}\,\left(\Omega(z)-\Omega^*(w)\right)=0,\; z\ne\bar w.
\]
In particular,
\[
\begin{array}{l}
\cfrac{\Omega(z)-\Omega(0)}{z}=I-\Omega(0)\Omega(z),\quad
z\in\dC\setminus\{-\infty,-1]\cup[1,+\infty)\},\;z\ne
0,\\[4mm]
\Omega'(0)=I-\Omega(0)^2.
\end{array}
\]
This combined with \eqref{mobinn2} yields $\Lambda(z)=z
I_{\sD_{\Omega(0)}}$ in the representation $\eqref{mobinn}$ for an
inner function $\Omega\in\cRS(\sM)$.
\end{remark}

\section{Compressed resolvents and the class $\mathbf N_{\mathfrak M}^0[-1,1]$}
\begin{definition}
\label{Nev} Let $\sM$ be a Hilbert space. A $\bB(\sM)$-valued Nevanlinna function $M$ is said to
belong to the class $\bN_{\sM}^0[-1,1]$
 if it is holomorphic outside the interval $[-1,1]$  and
\[\lim_{\xi\to \infty}\xi M(\xi)=-I_\sM.
\]
\end{definition}

It follows from \cite{AHS2} that $M\in {\mathbf
N}^{0}_{\sM}[-1,1]$ if and only if there exist a Hilbert space $\sH$
containing $\sM$ as a subspace and a selfadjoint contraction $T$ in
$\sH$ such that $T$ is $\sM$-simple and
\[
 M(\xi)=P_\sM (T-\xi I)^{-1}\uphar\sM, \quad \xi\in \dC\setminus
 [-1,1].
\]
Moreover, formula \eqref{sch-fr} implies the following connections
between the classes ${\mathbf N}_{\sM}^0[-1,1]$ and $\cRS(\sM)$ (see
also \cite{AHS2, AHS3}):
\begin{equation}
\label{conmezhd}
\begin{array}{rl}
  M(\xi)\in \bN_{\sM}^0[-1,1] &\Longrightarrow\Omega(z):=
  M^{-1}(1/z)+1/z\in\cRS(\sM),\\[4mm]
 \Omega(z)\in \cRS(\sM)          &\Longrightarrow M(\xi):=\left(\Omega(1/\xi)-\xi\right)^{-1}\in \bN_{\sM}^0[-1,1].
\end{array}
\end{equation}
Let $\Omega(z)=(z I+D)(I+zD)^{-1}$ be an inner function from the class $\cRS(\sM),$ then by \eqref{conmezhd}
\[
 \Omega(z)=(z I+D)(I+zD)^{-1}\Longrightarrow M(\xi)=\cfrac{\xi I+D}{1-\xi^2},\quad \xi\in\dC\setminus [-1,1].
\]
The identity $\Omega(z)^*\Omega(z)=I_\sM$ for $z\in\dT\setminus\{\pm
1\}$ is equivalent to
$$2\RE(\xi M(\xi))=-I_\sM,\quad \xi\in\dT\setminus\{\pm 1\}.$$

The next statement is established in \cite{Arl_arxiv_2017}. Here we give another proof.
\begin{theorem}
\label{ghtj} If $M(\xi)\in \bN_{\sM}^0[-1,1]$, then the function
\[
\cfrac{M^{-1}(\xi)}{\xi^2-1},\quad \xi\in \dC\setminus[-1,1],
\]
belongs to $\bN_{\sM}^0[-1,1]$ as well.
\end{theorem}
\begin{proof}
Let $M(\xi)\in \bN_{\sM}^0[-1,1]$. Then due to \eqref{conmezhd}
 the function $\Omega(z)= M^{-1}(1/z)+1/z$ belongs to $\cRS(\sM)$.
By Theorem \ref{newchar} the function
\[
\Upsilon(z)=\left(zI-\Omega(z)\right)\left(I-z\Omega(z)\right)^{-1},\quad 
 z\in\dC\setminus\{(-\infty,-1]\cup[1,+\infty)\}
\]
belongs to $\cRS(\sM).$ From the equality
\[
I-z\Upsilon(z)=(1-z^2)\left(I-z\Omega(z)\right)^{-1},\quad
z\in\dC\setminus\{(-\infty,-1]\cup[1,+\infty)\}
\]
we get
\[
\left(I-z\Upsilon(z)\right)^{-1}=\cfrac{I-z\Omega(z)}{1-z^2}.
\]
Simple calculations give
\[
\left(\Upsilon(1/\xi)-\xi\right)^{-1}=\cfrac{M^{-1}(\xi)}{\xi^2-1},\quad
\xi\in\dC\setminus[-1,1].
\]
Now in view of \eqref{conmezhd} the function
$\cfrac{M^{-1}(\xi)}{\xi^2-1}$ belongs to $\bN_{\sM}^0[-1,1]$.
\end{proof}

\section{Transformations of the classes $\mathcal R\mathcal S(\mathfrak M)$ and
$\mathbf N_{\mathfrak M}^0[-1,1]$}\label{sec6}

We start by studying transformations of the class $\mathcal R\mathcal S(\mathfrak M)$ given
by \eqref{TREIN}, \eqref{mappings2}:
\[
\cR\cS(\sM)\ni\Omega\mapsto{\bf\Phi}(\Omega)=\Omega_{\bf\Phi}(z):=(zI-\Omega(z))(I-z\Omega(z))^{-1},
\]
\[
\cR\cS(\sM)\ni\Omega\mapsto{\bf\Xi}_a(\Omega)=\Omega_a
(z):=\Omega\left(\cfrac{z+a}{1+za}\right),\quad a\in(-1,1),
\]
and the transform
\begin{equation}\label{TRZWEI}
\cR\cS(H)\ni\Omega\mapsto{\bf \Pi}(\Omega)=\Omega_{\bf\Pi}
(z): K_{11}+K_{12}\Omega(z)(I-K_{22}\Omega(z))^{-1}K^*_{12},
\end{equation}
which is determined by the selfadjoint contraction $K$ of the form
\[
{\bf K}=\begin{bmatrix}
K_{11}&K_{12}\cr K^*_{12}& K_{22}\end{bmatrix}:
\begin{array}{l}\sM\\\oplus\\ H\end{array}\to \begin{array}{l}\sM\\\oplus\\
H\end{array};
\]
in all these transforms
$z\in\dC\setminus\{(-\infty,-1]\cup[1,+\infty)\}$.

A particular case of \eqref{TRZWEI} is the transformation ${\bf
\Pi}_a$ determined by the block operator
\[
{\bf K}_a=\begin{bmatrix}aI &\sqrt{1-a^2}I\cr \sqrt{1-a^2}&- a I
\end{bmatrix}:\begin{array}{l}\sM\\\oplus\\ \sM\end{array}\to
\begin{array}{l}\sM\\\oplus\\ \sM\end{array}, \quad a\in(-1,1),
\]
i.e., see \eqref{mappings2},
\[
\cR\cS(\sM)\ni\Omega(z)\mapsto\wh\Omega_a(z):=(aI+\Omega(z))(I+a\Omega(z))^{-1}.
\]
By Theorem \ref{newchar} the mapping ${\bf\Phi}$ given by
\eqref{TREIN} is an automorphism of the class $\cR\cS(\sM)$,
${\bf\Phi}^{-1}={\bf\Phi}.$ The equality \eqref{forminner} shows
that the set of all inner functions of the class $\cRS(\sM)$ is the
image of all constant functions under the transformation ${\bf
\Phi}$. In addition, for $a,b\in(-1,1)$ the following identities
hold:
\[
{\bf \Pi}_b\circ{\bf \Pi}_a= {\bf \Pi}_a\circ{\bf \Pi}_b={\bf
\Pi}_c,\quad {\bf \Xi}_b\circ{\bf \Xi}_a= {\bf \Xi}_a\circ{\bf
\Xi}_b={\bf \Xi}_c,\; \mbox{ where }\; c={\cfrac{a+b}{1+ab}}.
\]
The mapping ${\bf \Gamma}$ on the class $\bN_{\sM}^0[-1,1]$ (see
Theorem \ref{ghtj}) defined by
\begin{equation}\label{GTR1}
\bN_{\sM}^0[-1,1]\ni M(\xi)\stackrel{{{\bf \Gamma}}}{\mapsto}M_{{\bf
\Gamma}}(\xi):=\cfrac{M^{-1}(\xi)}{\xi^2-1}\in \bN_{\sM}^0[-1,1]
\end{equation}
has been studied recently in \cite{Arl_arxiv_2017}. It is obvious
that ${{\bf \Gamma}}^{-1}={{\bf \Gamma}}$.

Using the relations \eqref{conmezhd} we define the transform ${\bf
U}$ and its inverse ${\bf U}^{-1}$ which connect the classes
$\cR\cS(\sM)$ and $\bN_{\sM}^0[-1,1]$:
\begin{equation}\label{UTR1}
\cR\cS(\sM)\ni\Omega(z)\stackrel{{{\bf
U}}}{\mapsto}M(\xi):=\left(\Omega(1/\xi)-\xi\right)^{-1}\in
\bN_{\sM}^0[-1,1], \quad \xi\in\dC\setminus [-1,1].
\end{equation}
\begin{equation}\label{UTR11}
\bN_{\sM}^0[-1,1]\ni M(\xi)\stackrel{{{\bf U}^{-1}}}{\mapsto}\Omega(z):= M^{-1}(1/z)+1/z\in\cRS(\sM),
\end{equation}
where $z\in\dC\setminus\{(-\infty,-1]\cup [1,+\infty)\}.$ The proof
of Theorem \ref{ghtj} contains the following commutation relations
\begin{equation} \label{comrel}
{\bf U}{{\bf\Phi}}={\bf \Gamma}{\bf U},\quad {\bf\Phi}{\bf
U}^{-1}={\bf U}^{-1}{{\bf\Gamma}}.
\end{equation}
One of the main aims in this section is to solve the following
realization problem concerning the above transforms: given a passive
selfadjoint system $\tau=\{T;\sM,\sM,\cK\}$ with the transfer
function $\Omega$, construct a passive selfadjoint systems whose
transfer function coincides with $\Phi(\Omega)$,
${\bf\Xi}_a(\Omega)$, ${\bf \Pi}(\Omega)$, and ${\bf\Pi}_a(\Omega)$,
respectively. We will also determine the fixed points of all the
mappings ${\bf\Phi}$, ${\bf\Gamma}$, ${\bf\Xi}_a$, and ${\bf\Pi}_a$.

\subsection{The mappings ${\bf \Phi}$ and ${\bf\Gamma}$ and inner dilations of the functions from $\mathcal R\mathcal S(\mathfrak M)$}

\begin{theorem}
\label{neueth1}

\begin{enumerate}
\item Let
$\tau=\left\{T;\sM,\sM,\cK\right\}$ be a passive selfadjoint system
and let $\Omega$ be its transfer function. Define
\begin{equation}
\label{wollen}
T_{\bf\Phi}:=\begin{bmatrix}-P_\sM T\uphar\sM&P_\sM D_T\cr D_T\uphar\sM &T \end{bmatrix}:\begin{array}{l}\sM\\\oplus\\\sD_T\end{array}\to\begin{array}{l}\sM\\\oplus\\\sD_T\end{array}.
\end{equation}
Then $T_{\bf\Phi}$ is a selfadjoint contraction and
$\Omega_{\bf\Phi}(z)=(zI-\Omega(z))(I-z\Omega(z))^{-1}$ is the
transfer function of the passive selfadjoint system of the form
\[
\tau_{\bf\Phi}=\left\{T_{\bf\Phi}; \sM,\sM, \sD_T\right\}.
\]
Moreover, if the system $\tau$ is minimal, then the system
$\tau_{\bf\Phi}$ is minimal, too.

\item Let $T$ be a selfadjoint contraction in $\sH$, let $\sM$ be a subspace of $\sH$ and let
\begin{equation}\label{wollen1}
M(\xi)=P_\sM(T-\xi I)^{-1}\uphar\sM.
\end{equation}
Consider a Hilbert space $\wh \sH:=\sM\oplus\sH$ and let $\wh P_\sM$
be the orthogonal projection in $\wh \sH$ onto $\sM$. Then
\[
\cfrac{M^{-1}(\xi)}{\xi^2-1}=\wh P_\sM(T_{\bf\Phi}-\xi I)^{-1}\uphar\sM,
\]
where $T_{\bf\Phi}$ is defined by \eqref{wollen}.

\item The function
\[
\wt\Omega(z)=(zI-T_{\bf\Phi})(I-z T_{\bf\Phi})^{-1}, \quad
 z\in\dC\setminus\{(-\infty,-1]\cup[1,+\infty)\}
\]
satisfies
\[
\Omega(z)=P_\sM\wt\Omega(z)\uphar\sM.
\]
\end{enumerate}
\end{theorem}
\begin{proof}
(1) According to \eqref{sch-fr} one has
$$P_\sM(I-zT)^{-1}\uphar\sM=(I_\sM-z\Omega(z))^{-1}$$
 for $z\in\dC\setminus\{(-\infty,-1]\cup[1,+\infty)\}$.
Let
\[
\Omega_{\bf\Phi}(z)=(zI-\Omega(z))(I-z\Omega(z))^{-1}.
\]
Now simple calculations give
\begin{equation}
\label{omgam}
\Omega_{\bf\Phi}(z)=\left(z-\cfrac{1}{z}\right)\left(I-z\Omega(z)\right)^{-1}+\cfrac{I_\sM}{z}
=P_\sM(zI-T)(I-zT)^{-1}\uphar\sM.
\end{equation}
Observe that the subspace $\sD_T$ is invariant under $T$; cf.
\eqref{commut}. Let $\sH:=\sM\oplus\sD_T$ and let $T_{\bf\Phi}$ be
given by \eqref{wollen}. Since $T$ is a selfadjoint contraction in
$\sM\oplus\cK$, the inequalities
\[
\left(\begin{bmatrix}\varphi\cr f\end{bmatrix},
\begin{bmatrix}\varphi\cr f\end{bmatrix}\right)\pm
\left(\begin{bmatrix}\varphi\cr f\end{bmatrix},
T_{\bf\Phi}\begin{bmatrix}\varphi\cr f\end{bmatrix}\right)=
\left\|(I\mp T)^{1/2}\f\pm (I\pm T)^{1/2}f \right\|^2
\]
hold for all $\varphi\in \sM$ and $f\in\sD_T$. Therefore
$T_{\bf\Phi}$ is a selfadjoint contraction in the Hilbert space
$\sH$ and the system
\[
\tau_{\bf\Phi}=\left\{\begin{bmatrix}-P_\sM T\uphar\sM&P_\sM D_T\cr D_T\uphar\sM &T \end{bmatrix}; \sM,\sM, \sD_T\right\}
\]
is passive selfadjoint. Suppose that $\tau$ is minimal, i.e.,
$$\cspan\{T^n\sM,\;n\in\dN_0\}=\sM\oplus\cK\Longleftrightarrow\bigcap\limits_{n=0}^\infty\ker(P_\sM T^n)=\{0\}. $$
Since
\[
\sD_T\ominus\{\cspan\{T^nD_T\sM,\;n\in\dN_0\}\}=\bigcap\limits_{n=0}^\infty\ker(P_\sM T^nD_T),
\]
we get $\cspan\{T^nD_T\sM:\;n\in\dN_0\}=\sD_T.$ This means that the
system $\tau_\Gamma$ is minimal.

For the transfer function $\Upsilon(z)$ of $\tau_{\bf\Phi}$ we
get
\[
\begin{array}{ll}
\Upsilon(z)& =\; (-P_\sM T+zP_\sM D_T(I-z T)^{-1}D_T)\uphar\sM\\[3mm]
 &=\; P_\sM\left(-T+zD^{2}_T(I-zT)^{-1}\right)\uphar\sM\\[3mm]
 &=\; P_\sM(zI-T)(I-zT)^{-1}\uphar\sM,
\end{array}
\]
with $z\in\dC\setminus\{(-\infty,-1]\cup[1,+\infty)\}$. Comparison
with \eqref{omgam} completes the proof.

(2) The function $M(\xi)=P_\sM(T-\xi I)^{-1}\uphar\sM$ belongs to
the class $\bN_{\sM}^0[-1,1].$ Consequently, $\Omega(z):=
M^{-1}(1/z)+1/z\in\cRS(\sM)$. The function $\Omega$ is the transfer
function of the passive selfadjoint system
$$\tau=\left\{T;\sM,\sM,\cK\right\},$$
where $\cK=\sH\ominus\sM$. Let $\Upsilon={\bf \Phi}(\Omega)$ and
$\wh M={\bf U}(\Upsilon)$. From \eqref{GTR1}--\eqref{comrel} it
follows that
$$\wh M(\xi)=\cfrac{M^{-1}(\xi)}{\xi^2-1},\quad \xi\in\dC\setminus[-1,1].$$
As was shown above, the function $\Upsilon$ is the transfer function
of the passive selfadjoint system
\[
\tau_{\bf\Phi}=\left\{T_{\bf\Phi}; \sM,\sM, \sH\right\},
\]
where $T_{\bf\Phi}$ is given by \eqref{wollen}. Then again the
Schur-Frobenius formula \eqref{sch-fr} gives
$$\wh M(\xi)=\wh P_\sM(T_{\bf\Phi}-\xi I)^{-1}\uphar\sM,\quad \xi\in\dC\setminus[-1,1].$$

(3) For all $z\in\dC\setminus\{(-\infty,-1]\cup[1,+\infty)\}$ one
has
\[
\wt\Omega(z)=\left(z-\cfrac{1}{z}\right)(I-zT_{\bf\Phi})^{-1}+\cfrac{1}{z} \,I.
\]
Then
\[
\begin{split}
P_\sM\wt\Omega(z)\uphar\sM
&=\left(z-\cfrac{1}{z}\right)(I_\sM-z\Upsilon(z))^{-1}+\cfrac{1}{z} \,I_\sM\\
&=(zI_\sM-\Upsilon(z))(I_\sM-z\Upsilon(z))^{-1}=\Omega(z).
\end{split}
\]
This completes the proof.
\end{proof}

Notice that if $\Omega(z)\equiv const=D$, then $\Upsilon
(z)=(zI-D)(I-zD)^{-1}$,
$z\in\dC\setminus\{(-\infty,-1]\cup[1,+\infty)\}$. This is the
transfer function of the conservative and selfadjoint system
\[
\Sigma=\left\{\begin{bmatrix}-D&D_D\cr D_D&D\
\end{bmatrix},\sM,\sM,\sD_D\right\}.
\]
\begin{remark}
\label{dcgjv} The block operator $T_{\bf\Phi}$ of the form
\eqref{wollen} appeared in \cite{Arl_arxiv_2017} and relation
\eqref{wollen1} is also established in \cite{Arl_arxiv_2017}.
\end{remark}

\begin{theorem} \label{inerdil}
1) Let $\sM$ be a Hilbert space and let $\Omega\in\cRS(\sM)$. Then there exist a Hilbert space ${\wt\sM}$ containing $\sM$ as a subspace and a selfadjoint contraction $\wt A$ in $\wt \sM$ such that for all $z\in\dC\setminus\{(-\infty,-1]\cup[1,+\infty)\}$ the equality
\begin{equation}\label{ajhvelbk}
\Omega(z)=P_\sM(z I_{\wt\sM}+\wt A)(I_{\wt\sM}+z\wt A)^{-1}\uphar\sM
\end{equation}
holds. Moreover, the pair $\{\wt\sM,\wt A\}$ can be chosen such that
$\wt A$ is $\sM$-simple, i.e.,
\begin{equation}
\label{minim11}\cspan\{\wt A^n\sM:\;n\in\dN_0\}=\wt\sM.
\end{equation}
The function $\Omega$ is inner if and only if $\wt\sM=\sM$ in the
representation \eqref{minim11}.

If there are two representations of the form \eqref{ajhvelbk} with pairs $\{\wt \sM_1, \wt A_1\}$ and $\{\wt \sM_2, \wt A_2\}$  that are $\sM$-simple, then
there exists a unitary operator $\wt U\in\bB(\wt\sM_1,\wt\sM_2)$ such that
\begin{equation} \label{eybntrd}
\wt U\uphar\sM=I_\sM,\quad \wt A_2\wt U=\wt U\wt A_1.
\end{equation}

2) The formula
\begin{equation}\label{yjdbyn}
\Omega(z)=\int\limits_{-1}^{1}\cfrac{z+t}{1+zt}\,d\sigma(t),\quad
z\in\dC\setminus\{(-\infty,-1]\cup[1,+\infty)\},
\end{equation}
gives a one-one correspondence between functions $\Omega$ from the class $\cRS(\sM)$ and nondecreasing left-continuous $\bB(\sM)$-valued functions $\sigma$ on $[-1,1]$ with $\sigma(-1)=0,$ $\sigma(1)=I_\sM$.
\end{theorem}
\begin{proof}
1) Realize $\Omega$ as the transfer function of a minimal passive
selfadjoint system $\tau=\left\{T;\sM,\sM,\cK\right\}$. Let the
selfadjoint contraction $T_{\bf\Phi}$ be given by \eqref{wollen} and
let $\wt\sM:=\sM\oplus\sD_T$ and $\wt A:=-T_{\bf\Phi}$. Then the
relations \eqref{ajhvelbk} and \eqref{minim11} are obtained from
Theorem \ref{neueth1}. Using Proposition \ref{selfaconserv} one
concludes that $\Omega$ is inner precisely when $\wt\sM=\sM$ in the
righthand side of \eqref{minim11}. Since
\[
\begin{array}{l}
 P_\sM(z I_{\wt\sM_1}+\wt A_1)(I_{\wt\sM_1}+z\wt A_1)^{-1}\uphar\sM
 =P_\sM(z I_{\wt\sM_2}+\wt A_2)(I_{\wt\sM_2}+z\wt
 A_2)^{-1}\uphar\sM\\[3mm]
\Longleftrightarrow\; P_\sM(I_{\wt\sM_1}+z\wt
A_1)^{-1}\uphar\sM=P_\sM(I_{\wt\sM_2}+z\wt A_2)^{-1}\uphar\sM,
\end{array}
\]
the $\sM$-simplicity with standard arguments (see e.g.
\cite{AHS2,ArlKlotz2010}) yields the existence of unitary $\wt
U\in\bB(\wt\sM_1,\wt\sM_2)$ satisfying \eqref{eybntrd}.

2) Let \eqref{ajhvelbk} be satisfied and let $\sigma(t)=P_\sM\wt
E(t) \uphar\sM$, $t\in[-1,1]$, where $E(t)$ is the spectral family
of the selfadjoint contraction $\wt A$ in $\wt\sM$. Then clearly
\eqref{yjdbyn} holds.

Conversely, let $\sigma$ be a nondecreasing left-continuous
$\bB(\sM)$-valued function $[-1,1]$ with $\sigma(-1)=0,$
$\sigma(1)=I_\sM.$ Define $\Omega$ by the right-hand side of
\eqref{yjdbyn}. Then, the function $\Omega$ in \eqref{yjdbyn}
belongs to the class $\cRS(\sM)$.
\end{proof}

\begin{remark} \label{trach1}
If $\Omega$ is represented in the form \eqref{ajhvelbk}, then the
proof of Theorem \ref{neueth1} shows that the transfer function of
the passive selfadjoint system $\wt\sigma_{\bf \Phi}=\{(-\wt A)_{\bf
\Phi};\sM,\sM,\sD_{\wt A}\}$ coincides with $\Omega$. Moreover, if
$\wt A$ is $\sM$-simple, then $\wt \sigma_{\bf \Phi}$ is minimal.
\end{remark}

\begin{remark}
\label{lheu}
The functions from the class $\cS^{qs}(\sM)$ admits the following
integral representations, see \cite{AHS3}:
\[
\Theta(z)=\Theta(0)+z\,\int_{-1}^{1}\frac{1-t^2}{1-tz}\,dG(t),
\]
where $G(t)$ is a nondecreasing $\bB(\sM)$-valued function with
bounded variation, $G(-1)=0$, $G(1)\le I_\sM,$ and
\[
\left|\left(\left(\Theta(0)+\int_{-1}^{1}t\,dG(t)\right)f,g\right)\right|^2\le
\left(\left(I-G(1)\right)f,f\right)\,\left(\left(I-G(1)\right)g,g\right),
\quad f,g\in\sM.
\]
\end{remark}

\begin{proposition}[cf. \cite{Arl_arxiv_2017}]
\label{fixpoint}
1) The mapping ${\bf \Phi}$ of $\cR\cS(\sM)$ has a unique fixed point
\begin{equation}\label{fixfunc}
\Omega_0(z)=\cfrac{z I_\sM}{1+\sqrt{1-z^2}},\quad\mbox{with} \quad \Omega_0(i)=\cfrac{i I_\sM}{1+\sqrt{2}}.
\end{equation}
2) The mapping ${\bf \Gamma}$ has a unique fixed point
\begin{equation}\label{fixfunczw}
M_0(\xi)=-\cfrac{I_\sM}{\sqrt{\xi^2-1}}\quad \mbox{with}\quad M_0(i)=\cfrac{iI_\sM}{\sqrt{2}}.
\end{equation}
3) Define the weight function $\rho(t)$ and the weighted Hilbert
space $\sH_0$ as follows
\begin{equation}\label{vesighj} \begin{array}{l}\rho_0(t)=\cfrac{1}{\pi}\cfrac{1}{\sqrt{1-t^2}},\; t\in (-1,1),\\
\sH_0:=L_2([-1,1],\sM,\rho_0(t))=L_2([-1,1],\;\rho_0(t))\bigotimes\sM=\left\{f(t):\int\limits_{-1}^1\cfrac{||f(t)||^2_\sM}{\sqrt{1-t^2}}\, dt<\infty\right\}.
\end{array}
\end{equation}
Then $\sH_0$ is the Hilbert space with the inner product
$$\left(f(t),g(t)\right)_{\sH_0}=\cfrac{1}{\pi}\int\limits_{-1}^1(f(t),g(t))_\sM\,\rho_0(t)\,dt=\cfrac{1}{\pi}
\int\limits_{-1}^1\cfrac{(f(t),g(t))_\sM}{\sqrt{1-t^2}}\, dt.$$
Identify $\sM$ with a subspace of $\sH_0$ of constant
vector-functions $\{f(t)\equiv f,\;f\in\sM\}$. Let
$$\cK_0:=\sH_0\ominus\sM=\left\{f(t)\in\sH_0:\int\limits_{-1}^1\cfrac{(f(t),h)_\sM}{\sqrt{1-t^2}}\, dt=0\;\forall h\in\sM\right\}$$
and define in $\sH_0$ the multiplication operator by
\begin{equation}\label{ogthevy}
(T_0f)(t)=tf(t),\; f\in\sH_0.
\end{equation}
Then $\Omega_0(z)$ is the transfer function of the simple passive
selfadjoint system
$$\tau_0=\{ T_0;\sM,\sM,\cK_0\},$$
while
\[
M_0(\xi)=P_{\sM}(T_0-\xi I)^{-1}\uphar\sM.
\]
\end{proposition}
\begin{proof}
1)--2) Let $\Omega_0(z)$ be a fixed point of the mapping ${\bf\Phi}$ of $\cR\cS(\sM)$, i.e.,
\[
\Omega_0(z)=\left(zI-\Omega_0(z)\right)\left(I-z\Omega_0(z)\right)^{-1},\quad 
z\in\dC\setminus\{(-\infty,-1]\cup[1,+\infty)\}.
\]
Then
\[
(I-z\Omega_0(z))^2=(1-z^2)I_\sM.
\]
Using $\Omega_0\in\cR\cS(\sM)$ and the Taylor expansion
$\Omega_0(z)=\sum_{n=0}^\infty C_k z^k$ in the unit disk, it is seen
that $\Omega_0$ is of the form \eqref{fixfunc}.

It follows that the transform $M_0={\bf U}(\Omega_0)$ defined in \eqref{UTR1} is of the form
\eqref{fixfunczw} and it is the unique fixed point of the mapping
${{\bf\Gamma}}$ in \eqref{GTR1}; cf. \eqref{comrel}.

3) For each $h\in\sM$ straightforward calculations, see \cite[pages
545--546]{Ber}, lead to the equality
\[
-\cfrac{h}{\sqrt{\xi^2-1}}=\cfrac{1}{\pi}\int\limits_{-1}^1\cfrac{h}{t-\xi}\,\cfrac{1}{\sqrt{1-t^2}}\,dt.
\]
Therefore, if $T_0$ is the operator of the form \eqref{ogthevy}, then
\[
M_0(\xi)=P_{\sM}(T_0-\xi I)^{-1}\uphar\sM.
\]
It follows that $\Omega_0$ is the transfer function of the system
$\tau_0=\{ T_0;\sM,\sM,\cK_0\}.$
\end{proof}

As is well known, the Chebyshev polynomials of the first kind given
by
\[
\wh T_0(t)=1,\; \wh T_n(t):=\sqrt{2}\cos(n\arccos t),\;n\ge 1
\]
form an orthonormal basis of the space $L_2([-1,1],\rho_0(t)),$
where $\rho_0(t)$ is given by \eqref{vesighj}. These polynomials
satisfy the recurrence relations
\[
\begin{array}{l}
t\wh T_0(t)=\cfrac{1}{\sqrt{2}}\wh T_1(t),\quad
t\wh T_1(t)=\cfrac{1}{\sqrt{2}}\wh T_0(t)+\cfrac{1}{2}\wh T_2(t),\\
t\wh T_n(t)=\cfrac{1}{2}\wh T_{n-1}(t)+\cfrac{1}{2}\wh
T_{n+1}(t),\quad n\ne 2.
\end{array}
\]
Hence the matrix of the operator multiplication by the independent
variable in the Hilbert space $L_2([-1,1], \rho_0(t))$ w.r.t. the
basis $\{\wh T_n(t)\}_{n=0}^\infty$ (the Jacobi matrix) takes the
form
\[
J=\begin{bmatrix} 0 & \cfrac{1}{\sqrt{2}} & 0 &0   & 0 &
\cdot &
\cdot &\cdot \\
\cfrac{1}{\sqrt{2}} & 0 & \cfrac{1}{{2}} & 0 &0& \cdot &
\cdot &\cdot \\
0    & \cfrac{1}{{2}} & 0 & \cfrac{1}{{2}} &0& \cdot &
\cdot&\cdot   \\
0    & 0& \cfrac{1}{{2}} & 0 & \cfrac{1}{{2}} &0& \cdot &
\cdot
\\
\vdots & \vdots & \vdots & \vdots & \vdots & \vdots & \vdots&\vdots
\end{bmatrix}.
\]
In the case of vector valued weighted Hilbert space
$\sH_0=L_2([-1,1],\sM,\rho_0(t))$ the operator \eqref{ogthevy} is
unitary equivalent to the block operator Jacobi matrix ${\bf
J}_0=J\bigotimes I_\sM$. It follows that the function $\Omega_0$ is
the transfer function of the passive selfadjoint system with the
operator $T_0$ given by the selfadjoint contractive block operator
Jacobi matrix
\[
T_0=\left[\begin{array}{c|c}0&\begin{array}{cccc}\cfrac{1}{\sqrt{2}}I_\sM&0&0&\ldots
\end{array}\\
\hline\begin{array}{c}\cfrac{1}{\sqrt{2}}I_\sM\cr
0\cr\vdots\end{array}&{\bf J_0}\end{array}\right],\;\; {\bf
J_0}=\begin{bmatrix} 0 & \cfrac{1}{{2}}I_\sM & 0 &0   & 0 & \cdot &
\cdot &\cdot \\
\cfrac{1}{{2}}I_\sM & 0 & \cfrac{1}{{2}}I_\sM & 0 &0& \cdot &
\cdot &\cdot \\
0    & \cfrac{1}{{2}}I_\sM & 0 & \cfrac{1}{{2}}I_\sM &0& \cdot &
\cdot&\cdot   \\
0    & 0& \cfrac{1}{{2}}I_\sM & 0 & \cfrac{1}{{2}}I_\sM &0& \cdot &
\cdot
\\
\vdots & \vdots & \vdots & \vdots & \vdots & \vdots & \vdots&\vdots
\end{bmatrix}.
\]

  \subsection{The mapping ${\bf \Pi}$ and Redheffer product}
\begin{lemma}\label{hilfe}
Let $ H$ be a Hilbert space, let $K$ be a selfadjoint contraction in
$H$ and let $\Omega\in\cR\cS(H)$. If $||K||<1$, then
$(I-K\Omega(z))^{-1}$ is defined on $H$ and it is bounded for all
$z\in\dC\setminus\{(-\infty,-1]\cup[1,+\infty)\}$.
\end{lemma}
\begin{proof} If $|z|\le 1$, $z\ne\pm 1$, then $||K||<1$ and $||\Omega(z)||\le 1$
imply that $||K\Omega(z)||<1$. Hence $(I-K\Omega(z))^{-1}$ exists as
bounded everywhere defined operator on $H$.

Now let $|z|>1$ and $z\in
\dC\setminus\{(-\infty,-1]\cup[1,+\infty)\}$. Then there exists
$\beta\in (0,\pi/2)$ such that either $|z\sin\beta-i\cos\beta|=1$ or
$|z\sin\beta+i\cos\beta|=1$. Suppose that, for instance,
$|z\sin\beta-i\cos\beta|=1$. Then from \eqref{propert111} one
obtains $||\Omega(z)\sin\beta-i\cos\beta I_H||\le 1.$ Hence
$S:=\Omega(z)\sin\beta-i\cos\beta I_H$ satisfies $||S||\le 1$ and
one has
\[
\Omega(z)=\cfrac{S+i\cos\beta\, I_H}{\sin\beta}.
\]
Furthermore,
\[
\begin{array}{ll}
I-K\Omega(z)&=I-\cfrac{KS+i\cos\beta\,K
}{\sin\beta}=\dfrac{1}{\sin\beta}\left((\sin\beta
\,I-i\cos\beta\,K)-KS\right)\\[4mm]
 &= \dfrac{1}{\sin\beta}(\sin\beta \,I-i\cos\beta\,K)\left(I-(\sin\beta \,I-i\cos\beta\,K)^{-1}KS\right).
\end{array}
\]
Clearly
\[
||(\sin\beta
\,I-i\cos\beta\,K)^{-1}K||^2\le\cfrac{||K||^2}{\sin^2\beta+||K||^2\cos^2\beta}<1,
\]
which shows that $||(\sin\beta \,I-i\cos\beta\,K)^{-1}KS||<1$.
Therefore, the following inverse operator $\left(I-(\sin\beta
\,I-i\cos\beta\,K)^{-1}KS\right)^{-1}$ exists and is everywhere
defined on $H$. This implies that
\[
(I-K\Omega(z))^{-1}=\sin \beta\left(I-(\sin\beta \,I-i\cos\beta\,K)^{-1}KS\right)^{-1}(\sin\beta \,I-i\cos\beta\,K)^{-1}.
\]
\end{proof}

\begin{theorem}\label{mappi} Let
\[
{\bf K}=\begin{bmatrix}
K_{11}&K_{12}\cr K^*_{12}& K_{22}\end{bmatrix}:\begin{array}{l}\sM\\\oplus\\ H\end{array}\to \begin{array}{l}\sM\\\oplus\\ H\end{array}
\]
be a selfadjoint contraction. Then the following two assertions
hold:

1) If $||K_{22}||<1$, then for every $\Omega\in\cR\cS(H)$ the
transform
\begin{equation}\label{lkgo}
\Theta(z):=
 K_{11}+K_{12}\Omega(z)(I-K_{22}\Omega(z))^{-1}K^*_{12},\quad
 z\in\dC\setminus\{(-\infty,-1]\cup[1,+\infty)\},
\end{equation}
also belongs to $\cR\cS(\sM)$.

2) If $\Omega\in\cR\cS(H)$ and $\Omega(0)=0$, then again the
transform $\Theta$ defined in \eqref{lkgo} belongs to $\cR\cS(\sM)$.
\end{theorem}
\begin{proof} 1) It follows from Lemma \ref{hilfe} that $(I-K_{22}\Omega(z))^{-1}$ exists as a bounded operator on $H$
for all $z\in\dC\setminus\{(-\infty,-1]\cup[1,+\infty)\}$.
Furthermore,
\begin{multline*}
\Theta(z)-\Theta(z)^*=K_{12}\Omega(z)(I-K_{22}\Omega(z))^{-1}K^*_{12}-K_{12}(I-\Omega(z)^*K_{22})^{-1}\Omega(z)^*K^*_{12}\\
=K_{12}(I-\Omega(z)^*K_{22})^{-1}\left((I-\Omega(z)^*K_{22})\Omega(z)-\Omega(z)^*(I-K_{22}\Omega(z))\right)(I-K_{22}\Omega(z))^{-1}K^*_{12}\\
=K_{12}(I-\Omega(z)^*K_{22})^{-1}\left(\Omega(z)-\Omega(z)^*\right)(I-K_{22}\Omega(z))^{-1}K^*_{12}.
\end{multline*}
Thus, $\Theta$ is a Nevanlinna function on the domain
$\dC\setminus\{(-\infty,-1]\cup[1,+\infty)\}$.

Since ${\bf K}$ is a selfadjoint contraction, its entries are of the
form (again see Proposition \ref{ParContr} and Remark \ref{stut1}):
\[
K_{12}= ND_{K_{22}},\; K^*_{12}=D_{K_{22}}N^*,\; K_{11}=-NK_{22}N^*+D_{N^*}LD_{N^*},
\]
where $N:\sD_{K_{22}}\to \sM$ is a contraction and
$L:\sD_{N^*}\to\sD_{N^*}$ is a selfadjoint contraction. This gives
\[
 \Theta(z)=N\left(-K_{22}+D_{K_{22}}\Omega(z)(I-K_{22}\Omega(z))^{-1}D_{K_{22}}\right)N^*+D_{N^*}LD_{N^*}.
\]
Denote
\[
\wt\Theta(z):=-K_{22}+D_{K_{22}}\Omega(z)(I-K_{22}\Omega(z))^{-1}D_{K_{22}}.
\]
Then
\[
 \wt\Theta(z)=D^{-1}_{K_{22}}(\Omega(z)-K_{22})(I-K_{22}\Omega(z))^{-1}D_{K_{22}}=D_{K_{22}}(I-\Omega(z)K_{22})^{-1}(\Omega(z)-K_{22})D^{-1}_{K_{22}}
\]
and
\[
\Theta(z)=N\wt\Theta(z)N^*+D_{N^*}LD_{N^*}.
\]
Again straightforward calculations (cf. \cite{Shmul1,AHS2007}) show
that for all $f\in\sD_{K_{22}}$,
\[
 ||f||^2-||\wt\Theta(z)f||^2 
  =||(I-K_{22}\Omega(z))^{-1}D_{K_{22}}f||^2-||{\Omega(z)}(I-K_{22}\Omega(z))^{-1}D_{K_{22}}f||^2,
\]
and for all $h\in\sM$,
\begin{multline*}
||h||^2-||\Theta(z)h||^2\\
=||N^*h||^2-||\wt\Theta(z)N^*h||^2+||D_LD_{N^*}h||^2+||(D_{N}\wt
\Theta(z)N^*-N^*LD_{N^*})h||^2.
\end{multline*}
Since $\Omega(z)$ is a contraction for all $|z|\le 1$, $z\ne\pm 1$,
one concludes that $\wt \Theta(z)$ and, thus, also $\Theta(z)$ is a
contraction. In addition, the operators $\Theta(x)$ are selfadjoint
for $x\in(-1,1)$. Therefore $\Theta \in\cR\cS(\sM)$.

2) Suppose that $\Omega(0)=0$. To see that the operator
$(I-K_{22}\Omega(z))^{-1}$ exists as a bounded operator on $H$ for
all $z\in\dC\setminus\{(-\infty,-1]\cup[1,+\infty)\}$, realize
$\Omega$ as the transfer function of a passive selfadjoint system
\[
\sigma=\left\{\begin{bmatrix}0&N\cr N^*& S \end{bmatrix};H,H,\cK\right\},
\]
i.e., $\Omega(z)=zN(I-zS)^{-1}N^*$.
Since
$$T=\begin{bmatrix}0&N\cr N^*& S \end{bmatrix}:\begin{array}{l}H\\\oplus\\\cK\end{array}\to \begin{array}{l}H\\\oplus\\\cK\end{array}$$
is a selfadjoint contraction, the operator $N\in\bB(\cK,H)$ is a
contraction and $S$ is of the form $S=D_{N^*}LD_{N^*}$, where
$L\in\bB(\sD_{N^*})$ is a selfadjoint contraction. It follows that
the operator $N^*K_{22}N+S$ is a selfadjoint contraction for an
arbitrary selfadjoint contraction $K_{22}$ in $H$. Therefore,
$\left(I-z(N^*K_{22}N+S)\right)^{-1}$ exists on $\cK$ and is bounded
for all $z\in\dC\setminus\{(-\infty,-1]\cup[1,+\infty)\}$. It is
easily checked that for all
$z\in\dC\setminus\{(-\infty,-1]\cup[1,+\infty)\}$ the equality
\[
\left(I-zK_{22}N(I-zS)^{-1}N^*\right)^{-1}=I+zK_{22}N\left(I-z(N^*K_{22}N+S)\right)^{-1}N^*
\]
holds. Now arguing again as in item 1) one completes the proof.
\end{proof}

\begin{theorem} \label{redpro}
Let
\[
{\bf S}=\begin{bmatrix}A&B\cr B^*&G
\end{bmatrix}:\begin{array}{l}H\\ \oplus\\ \cK\end{array}\to
\begin{array}{l}H\\ \oplus\\ \cK\end{array},
 \quad {\bf K}=\begin{bmatrix} K_{11}&K_{12}\cr K^*_{12}&
K_{22}\end{bmatrix}:\begin{array}{l}\sM\\\oplus\\ H\end{array}\to
\begin{array}{l}\sM\\\oplus\\ H\end{array}
\]
be selfadjoint contractions. Also let $\sigma=\{{\bf S}, H,H,\cK\}$
be a passive selfadjoint system with the transfer function
$\Omega(z)$. Then the following two assertions hold:

1) Assume that $||K_{22}||<1$. Then $\Theta(z)$ given by
\eqref{lkgo} is the transfer function of the passive selfadjoint
system
\[
\tau=\{{\bf T}, \sM,\sM,\cK\},
\]
where ${\bf T}={\bf K}\bullet{\bf S}$ is the Redheffer product (see
\cite{Redheffer, Timotin1995}):
 \begin{equation}
 \label{entredh}
{\bf T}=\begin{bmatrix}K_{11}+K_{12}A(I-K_{22}A)^{-1}K_{12}^*& K_{12}(I-AK_{22})^{-1}B\cr
B^*(I-K_{22}A)^{-1}K_{12}^*& G+B^*K_{22}(I-AK_{22})^{-1}B  \end{bmatrix}: \begin{array}{l}\sM\\\oplus\\ \cK\end{array}\to \begin{array}{l}\sM\\\oplus\\ \cK\end{array}.
\end{equation}
2) Assume that $A=0$. Then the Redheffer product ${\bf T}={\bf
K}\bullet{\bf S}$ is given by
\[
{\bf T}=\begin{bmatrix}K_{11}& K_{12}B\cr B^*K_{12}^*& G+B^*K_{22}B
\end{bmatrix}: \begin{array}{l}\sM\\\oplus\\ \cK\end{array}\to
\begin{array}{l}\sM\\\oplus\\ \cK\end{array}
\]
and the transfer function of the passive selfadjoint system
$\tau=\{{\bf T}, \sM,\sM,\cK\}$ is equal to the function $\Theta$
defined in \eqref{lkgo}.
\end{theorem}
\begin{proof}
By definition
\[
\Omega(z)=A+zB(I-zG)^{-1}B^*,\quad
z\in\dC\setminus\{(-\infty,-1]\cup[1,+\infty)\}.
\]
1) Suppose that $||K_{22}||<1$. Since
\[
\Theta(z)= K_{11}+K_{12}\Omega(z)(I-K_{22}\Omega(z))^{-1}K_{12}^*=K_{11}+K_{12}(I-\Omega(z)K_{22})^{-1}\Omega(z)K_{12}^*,
\]
one obtains
\begin{multline*}
\Theta(z)-\Theta(0)=K_{12}(I-\Omega(z)K_{22})^{-1}\left(\Omega(z)-\Omega(0)\right)(I-K_{22}\Omega(0))^{-1}K_{12}^*\\
=zK_{12}\left(I-AK_{22}-zB(I-zG)^{-1}B^*K_{22}\right)^{-1}B(I-zG)^{-1}B^*(I-K_{22}A)^{-1}K_{12}^*.
\end{multline*}
Furthermore,
\[
\begin{array}{l}
\left(I-AK_{22}-zB(I-zG)^{-1}B^*K_{22}\right)^{-1}B(I-zG)^{-1}\\
\qquad=(I-AK_{22})^{-1}\left(I-zB(I-zG)^{-1}B^*K_{22}(I-AK_{22})^{-1}\right)^{-1}B(I-zG)^{-1}\\
\qquad=(I-AK_{22})^{-1}B\left(I-z(I-zG)^{-1}B^*K_{22}(I-AK_{22})^{-1}B\right)^{-1}(I-zG)^{-1}\\
\qquad=(I-AK_{22})^{-1}B\left(I-z\left(G+zB^*K_{22}(I-AK_{22})^{-1}B\right)\right)^{-1}
\end{array}
\]
and one has
\begin{multline*}
\Theta(z)= K_{11}+K_{12}A(I-K_{22}A)^{-1}K_{12}^*\\
+zK_{12}(I-AK_{22})^{-1}B\left(I-z\left(G+zB^*K_{22}(I-AK_{22})^{-1}B\right)\right)^{-1}B^*(I-K_{22}A)^{-1}K_{12}^*.
\end{multline*}
Now it follows from \eqref{entredh} that $\Theta(z)$ is the transfer
function of the system $\tau$.

Next it is shown that the selfadjoint operator ${\bf T}$ given by
\eqref{entredh} is a contraction. Let the entries of ${\bf S}$ and
${\bf K}$ be parameterized by
\[
\left\{\begin{array}{l} B^*=UD_{A},B=D_{A}U^*\\
G=-UAU^*+D_{U^*}ZD_{U^*}\end{array},\right. \quad
\left\{\begin{array}{l}K_{12}=VD_{K_{22}},K_{12}^*=D_{K_{22}}V^*\\
K_{11}=-VK_{22}V^*+D_{V^*}YD_{V^*}\end{array}\right.,
\]
where $V,U,Y,Z$ are contractions acting between the corresponding
subspaces. Also define the operators
\[
\begin{array}{l}\Phi_{K_{22}}(A)=-K_{22}+D_{K_{22}}A(I-K_{22}A)^{-1}D_{K_{22}},\\[3mm]
 \Phi_A(K_{22})=-A+D_{A}K_{22}(I-AK_{22})^{-1}D_{A}.
 \end{array}
\]
This leads to the formula
\begin{multline*}
{\bf T}=\begin{bmatrix}V&0\cr 0&U\end{bmatrix}\begin{bmatrix}\Phi_{K_{22}}(A)&D_{K_{22}}(I-AK_{22})^{-1}D_A\cr
D_A(I-K_{22}A)^{-1}D_{K_{22}}&\Phi_A(K_{22})\end{bmatrix} \begin{bmatrix}V^*&0\cr 0&U^*\end{bmatrix}\\
+\begin{bmatrix}D_{V^*}YD_{V^*}&0\cr 0&D_{U^*}ZD_{U^*}\end{bmatrix}.
\end{multline*}
The block operator
\[
{\dJ}=\begin{bmatrix}\Phi_{K_{22}}(A)&D_{K_{22}}(I-AK_{22})^{-1}D_A\cr
D_A(I-K_{22}A)^{-1}D_{K_{22}}&\Phi_A(K_{22})\end{bmatrix}
\]
is unitary and selfadjoint. Actually, the selfadjointness follows
from selfadjointness of the operators $A, K_{22}$ and
$\Phi_{K_{22}}(A), \Phi_{A}(K_{22}).$ Furthermore, one has the
equalities
\[
 \begin{array}{l} ||f||^2-||\Phi_{K_{22}}(A)f||^2=||D_A(I-K_{22}A)^{-1}D_{K_{22}}f||^2,\\[3mm]
 ||g||^2-||\Phi_{A}(K_{22})g||^2=||D_{K_{22}}(I-AK_{22})^{-1}D_{A}g||^2,\\[3mm]
 \left(\Phi_{K_{22}}(A)f, D_{K_{22}}(I-AK_{22})^{-1}D_{A}g\right)=\left(D_A(I-K_{22}A)^{-1}(A-K_{22})(I-K_{22}A)^{-1}D_{K_{22}}f,g \right),\\[3mm]
 \left(\Phi_{A}(K_{22})g,
   D_{A}(I-K_{22}A)^{-1}D_{K_{22}}f\right)=\left(D_{K_{22}}(I-AK_{22})^{-1}(K_{22}-A)(I-AK_{22})^{-1}D_{A}g,f
\right).
\end{array}
\]
These equalities imply that $\dJ$ is unitary.

Denote
\[
\dW=\begin{bmatrix}V&0\cr 0&U\end{bmatrix},\quad
\dX=\begin{bmatrix}Y&0\cr 0&Z\end{bmatrix}.
\]
Then
\[
\bT=\dW\dJ\dW^*+D_{\dW^*}\dX D_{\dW^*},
\]
and one obtains the equality
\[
 ||h||^2-||\bT h||^2=||D_{\dX}D_{\dW^*}h||^2+||(\dW^*\dX-D_{\dW}\dJ\dW^*)h||^2.
\]
Thus, $\bT$ is a selfadjoint contraction.

The proof of the statement 2) is similar to the proof of statement
1) and is omitted.
\end{proof}

\subsection{The mapping $\Omega(z)\mapsto(a\,I+\Omega(z))\left(I+a\,\Omega(z)\,\right)^{-1}$}
\begin{proposition}\label{Fruhlingnew}
Let
$$\tau=\left\{\begin{bmatrix}A&B\cr B^*&G \end{bmatrix};\sM,\sM,\cK\right\}$$
be a passive selfadjoint system with transfer function $\Omega.$ Let
$a\in(-1,1)$. Then the passive selfadjoint system
\[
\sigma_a=\left
\{\begin{bmatrix}(aI+A)(I+aA)^{-1}&\sqrt{1-a^2}(I+aA)^{-1}B\cr \sqrt{1-a^2} B^*(I+aA)^{-1}&G-aB^*(I+aA)^{-1}B \end{bmatrix};\sM,\sM,\cK\right\}
\]
has transfer function
\[
\wh\Omega_a(z)=(a\,I+\Omega(z))(I+a\,\Omega(z))^{-1},\quad
z\in\dC\setminus\{(-\infty,-1]\cup[1,+\infty)\}.
\]
\end{proposition}
\begin{proof}
Let
\[
{\bf K}_a=\begin{bmatrix}aI &\sqrt{1-a^2}I\cr
\sqrt{1-a^2}&- a I \end{bmatrix}:\begin{array}{l}\sM\\\oplus\\ \sM\end{array}\to \begin{array}{l}\sM\\\oplus\\ \sM\end{array},\; 
{\bf S}=\begin{bmatrix}A&B\cr B^*&G \end{bmatrix}:\begin{array}{l}\\\sM\\ \oplus\\ \cK\end{array}\to \begin{array}{l}\sM\\ \oplus\\ \cK\end{array}.
\]
Then the Redheffer product ${\bf K}_a\bullet {\bf S}$ (cf.
\eqref{entredh}) takes the form
\begin{equation}\label{xfcnyck}
{\bf T}= \begin{bmatrix}(aI+A)(I+aA)^{-1}&\sqrt{1-a^2}(I+aA)^{-1}B\cr \sqrt{1-a^2} B^*(I+aA)^{-1}&G-aB^*(I+aA)^{-1}B \end{bmatrix}:
\begin{array}{l}\sM\\\oplus\\\cK\end{array}\to \begin{array}{l}\sM\\\oplus\\\cK\end{array}.
\end{equation}
On the other hand, for all
$\;z\in\dC\setminus\{(-\infty,-1]\cup[1,+\infty)\}$ one has
\[
\begin{array}{ll}
 K_{11}+K_{12}\Omega(z)(I-K_{22}\Omega(z))^{-1}K^*_{12}
 &=aI+(1-a^2)\Omega(z)(I+a\Omega(z))^{-1}\\
 &=(a\,I+\Omega(z))(I+a\,\Omega(z))^{-1}.
\end{array}
\]
This completes the proof.
\end{proof}

\subsection{The mapping $\Omega(z)\mapsto\Omega\left(\cfrac{z+a}{1+za}\right)$ and its fixed points}
For a contraction $S$ in a Hilbert space and a complex number $a,$
$|a|<1$, define, see \cite{SF},
\[
 S_a:=(S-aI)(I-\bar aS)^{-1}.
\]
The operator $S_a$ is a contraction, too. If $S$ is a selfadjoint
contraction and $a\in(-1,1)$, then $S_a$ is also selfadjoint. One
has $S_a=W_{-a}(S)$ (see Introduction) and, moreover,
\begin{equation}\label{resosa}
\begin{array}{l}
 D_{S_a}=\sqrt{1-a^2}(I-aS)^{-1}D_S,\\
 (I-zS_a)^{-1}=\cfrac{1}{1+az}(I-a S)\left(I-\cfrac{z+a}{1+az}S\right)^{-1},\\
 (zI-S_a)(I-zS_a)^{-1}=\left(\cfrac{z+a}{1+az}\,I-S\right)\left(I-\cfrac{z+a}{1+az}\,S\right)^{-1},
\end{array}
\end{equation}
where $z\in\dC\setminus\{(-\infty,-1]\cup[1,\infty\}$. Let the block
operator
\begin{equation}\label{Tblock2}
T=\begin{bmatrix}D&C\cr C^*&F
\end{bmatrix}:\begin{array}{l}\sM\\\oplus\\\cK
\end{array}\to\begin{array}{l}\sM\\\oplus\\\cK  \end{array}
\end{equation}
be a selfadjoint contraction and let $\Omega(z)=D+zC(I-zF)^{-1}C^*.$
Then from the Schur-Frobenius formula \eqref{Sh-Fr1}  and from the
relation
\[
T_a=(T-aI)(I-aT)^{-1}=\cfrac{1-a^2}{a}(I-aT)^{-1}-\cfrac{1}{a}\,I
\]
it follows that $T_a$ has the block form
\begin{equation}
\label{attrns}
{\SMALL\setlength{\arraycolsep}{2pt}
T_a=\begin{bmatrix}(\Omega(a)-aI)(I-a\Omega(a))^{-1}&(1-a^2)(I-a\Omega(a))^{-1}C(I-aF)^{-1}\cr (1-a^2) (I-aF)^{-1}C^*(I-a\Omega(a))^{-1}&F_a+a(1-a^2)(I-aF)^{-1}C^*(I-a\Omega(a))^{-1}C(I-aF)^{-1}  \end{bmatrix}}
\end{equation}
 \begin{theorem}\label{ksia}
Let
$$\tau=\left\{\begin{bmatrix}D&C\cr C^*&F \end{bmatrix},\;\sM,\sM,\cK\right\}$$
be a passive selfadjoint system with the transfer function $\Omega$.
Then for every $a\in(-1,1)$ the $\bB(\sM)$-valued function
$$\Omega\left(\cfrac{z+a}{1+az}\right)$$
is the transfer function of the passive selfadjoint system
$$\tau_a=\left\{\begin{bmatrix}\Omega(a)&\sqrt{1-a^2}C(I-aF)^{-1}\cr \sqrt{1-a^2}(I-aF)^{-1}C^*&F_a \end{bmatrix},\;\sM,\sM,\cK\right\}.$$
Furthermore, if $\tau$ is a minimal system then $\tau_a$ is minimal,
too.
\end{theorem}

\begin{proof}
Let
$$C=KD_F, \;D=-KFK^*+D_{K^*}YD_{K^*},$$
be the parametrization for entries of the block operator $T$, cf.
\eqref{faul}, where  $K\in\bB(\sD_F,\cK)$ is a contraction and
$Y\in\bB(\sD_{K^*})$ is a selfadjoint contraction. From
\eqref{dshazom} and \eqref{resosa} we get
\[
\begin{array}{ll}
\Omega\left(\cfrac{z+a}{1+az}\right)
 &=\; D_{K^*}YD_{K^*}+K\left(\cfrac{z+a}{1+az}\,I-F\right)\left(I-\cfrac{z+a}{1+az}\,F\right)^{-1}K^*\\
 &=\; D_{K^*}YD_{K^*}+K\left(zI-F_a\right)\left(I-zF_a\right)^{-1}K^*
\end{array}
\]
with $z\in\dC\setminus\{(-\infty,-1]\cup[1,\infty\}$. The operator
\[
\begin{array}{ll}
 \wh T_a & = \begin{bmatrix}-KF_aK^*+D_{K^*}YD_{K^*}&KD_{F_a}\cr D_{F_a}K^*&F_a\end{bmatrix}\\
     & =\begin{bmatrix}\Omega(a)&\sqrt{1-a^2}C(I-aF)^{-1}\cr \sqrt{1-a^2}(I-aF)^{-1}C^*&F_a \end{bmatrix}
:\begin{array}{l}\sM\\\oplus\\\cK  \end{array}\to\begin{array}{l}\sM\\\oplus\\\cK  \end{array}
\end{array}
\]
is a selfadjoint contraction. The formula \eqref{dshazom} applied to
the system $\tau_a$ gives
\[
 \Omega_{\tau_a}(z)=D_{K^*}YD_{K^*}+K\left(zI-F_a\right)\left(I-zF_a\right)^{-1}K^*.
\]
Hence $\Omega_{\tau_a}(z)=\Omega\left(\cfrac{z+a}{1+az}\right)$ for
all $z\in\dC\setminus\{(-\infty,-1]\cup[1,\infty\}.$

Suppose $\tau$ is the minimal system. This is equivalent to the relations
\[
\begin{array}{ll}
 &\cspan\{F^nD_FK^*\sM:\;n\in\dN_0\}=\cK\\
 &\qquad \Longleftrightarrow \quad \bigcap\limits_{n=0}^\infty\ker (KF^nD_F)=\{0\}\\
 &\qquad \Longleftrightarrow \quad \bigcap_{|z|<1}\ker K(I-zF)^{-1}D_F=\{0\}.
\end{array}
\]
Using the formulas \eqref{resosa} one obtains
\[
\begin{array}{l}
\bigcap_{|z|<1}\ker K(I-zF_a)^{-1}D_{F_a}=\bigcap_{|z|<1}\ker
K\left(I-\cfrac{z+a}{1+az}F\right)^{-1}D_{F}(I-aF)\\[4mm]
=(I-aF)\bigcap_{|\mu|<1}\ker K(I-\mu F)^{-1}D_F=\{0\}
\end{array}
\]
or, equivalently,
\[
\cspan\{F^n_aD_{F_a}K^*\sM,\;n\in\dN_0\}=\cK.
\]
This shows that the system $\tau_a$ is minimal.
\end{proof}

\begin{remark}\label{herbst2}
1) Let $T$ in \eqref{Tblock2} be represented in the form
\[
T=\begin{bmatrix}K&0\cr 0&I
\end{bmatrix}\dJ_F\begin{bmatrix}K^*&0\cr 0&I
\end{bmatrix}+\begin{bmatrix}D_{K^*}YD_{K^*}&0\cr 0&0\end{bmatrix},
\]
see Remark \ref{herbst}. Then
\[
\begin{array}{ll}
\begin{bmatrix}-KF_aK^*+D_{K^*}YD_{K^*}&KD_{F_a}\cr D_{F_a}K^*&F_a \end{bmatrix}
 &=\; \begin{bmatrix}\Omega(a)&\sqrt{1-a^2}C(I-aF)^{-1}\cr \sqrt{1-a^2}(I-aF)^{-1}C^*&F_a
 \end{bmatrix}\\[5mm]
 &=\; \begin{bmatrix}K&0\cr 0&I \end{bmatrix}\dJ_{F_a}\begin{bmatrix}K^*&0\cr 0&I \end{bmatrix}
  +\begin{bmatrix}D_{K^*}YD_{K^*}&0\cr 0&0\end{bmatrix}.
\end{array}
\]
2) Let the transformation ${\bf V}_a$ with $a\in (-1,1)$ be defined
by
\[ \begin{bmatrix}D&C\cr C^*&F \end{bmatrix}\;\stackrel{{{\bf V}_a}}{\mapsto} \,\wh T_a
 =\begin{bmatrix}\Omega(a)&\sqrt{1-a^2}C(I-aF)^{-1}\cr \sqrt{1-a^2}(I-aF)^{-1}C^*&F_a
 \end{bmatrix}.
\]
Then for all $a,b\in (-1,1)$ one has the identities
\[
{\bf V}_a\circ{\bf V}_b={\bf V}_b\circ{\bf V}_a={\bf V}_c,\textrm{
 where } c=\cfrac{a+b}{1+ab}.
\]
\end{remark}
\begin{proposition}\label{nnnee}
The fixed points of the mapping
$\Omega(z)\mapsto\Omega\left(\cfrac{z+a}{1+za}\right)$,
$a\in(-1,1),$ $a\ne 0$, consist only of constant functions.
\end{proposition}
\begin{proof}
Suppose  that for some $a\in(-1,1),$ $a\ne 0$, the equality
$$\Omega\left(\cfrac{z+a}{1+az}\right)=\Omega(z)
$$
is satisfied for all $z\in\dC\setminus\{(-\infty,-1]\cup
[1,+\infty)\}$. Then, in particular, $\Omega(0)=\Omega(a)$.
Therefore from Theorem \ref{ksia} one obtains the equality
$KFK^*=KF_aK^*.$ Now
$$F-F_a= aD_F^2(I-aF)^{-1}$$
leads to
\[
(I-aF)^{-1/2}D_FK^*=0.
\]
Taking into account that $\ran K^*\subseteq\sD_F$, we get  $K^*=0$.
This means that $\Omega(z)\equiv \Omega(0)$. So, the fixed points of
the mapping $\Omega(z)\mapsto\Omega\left(\cfrac{z+a}{1+za}\right)$
are the constant functions only.
\end{proof}

\begin{remark}
A.~Filimonov and E.~Tsekanovski\u{\i} \cite{FilTsek1987} considered
$J$-unitary operator colligations that are automorphic invariant
w.r.t. a  subgroup $G$ of the M\"{o}bius transformations of the unit
disk and its representations in the channel and state spaces. The
characteristic function $W(z)$ of such a colligation satisfies the
condition
\[
W(g(z))V_g=V_g W(z),\quad \forall z\in\dD\quad\mbox{and}\quad
\forall g\in G,
\]
where $\{V_g\}$ is a representation of $G$ in the channel space.
\end{remark}

\subsection{The mapping $\Omega(z)\mapsto\left(\Omega\left(\cfrac{z+a}{1+az}\right)-a\,I\right)\left(I-a\,\Omega\left(\cfrac{z+a}{1+az}\right)\,\right)^{-1}
$ and its fixed points}
\begin{proposition}\label{prufung}
Let $\tau=\{T;\sM,\sM,\cK\}$ be a passive selfadjoint system with
transfer function $\Omega.$ Then the passive selfadjoint system
$\eta_a=\{T_a;\sM,\sM,\cK\}$, $a\in(-1,1)$, has the transfer
function
\[
\wt\Omega_a(z)=\left(\Omega\left(\cfrac{z+a}{1+az}\right)-a\,I_\sM\right)\left(I_\sM-a\,\Omega\left(\cfrac{z+a}{1+az}\right)\,\right)^{-1}.
\]
If $\tau$ is minimal then $\eta_a$ is minimal, too.
\end{proposition}
\begin{proof}
Let $T$ be a selfadjoint contraction in the Hilbert space $\sH$ and
let $a\in(-1,1)$. Due to \eqref{resosa} for all
$\;z\in\dC\setminus\{(-\infty,-1]\cup[1,\infty\}$ one has
\[
(I-zT_a)^{-1}=\cfrac{1}{1+az}(I-a
T)\left(I-\cfrac{z+a}{1+az}\,T\right)^{-1}.
\]
Moreover,
\[
\begin{array}{ll}
(I-a T)& \left(I-\cfrac{z+a}{1+az}\,T\right)^{-1}=\left(I-\cfrac{z+a}{1+az}\,T\right)^{-1}-aT\left(I-\cfrac{z+a}{1+az}\,T\right)^{-1}\\
 &=\left(I-\cfrac{z+a}{1+az}T\right)^{-1}+a\,\cfrac{1+za}{z+a}\,I-a\,\cfrac{1+za}{z+a}\left(I-\cfrac{z+a}{1+az}\,T\right)^{-1}\\
 &=a\,\cfrac{1+za}{z+a}\, I+\cfrac{z(1-a^2)}{z+a}\left(I-\cfrac{z+a}{1+az}\,T\right)^{-1},
\end{array}
\]
and
\[
\begin{array}{ll}
(I-zT_a)^{-1}
 &=\; \cfrac{1}{1+az}\left(a\,\cfrac{1+za}{z+a}\, I+\cfrac{z(1-a^2)}{z+a}\left(I-\cfrac{z+a}{1+az}\,T\right)^{-1}\right)\\
 &=\; \cfrac{a}{z+a}I+\cfrac{z(1-a^2)}{(z+a)(1+az)}\left(I-\cfrac{z+a}{1+az}\,T\right)^{-1}.
\end{array}
\]
Let $\sH=\sM\oplus\cK$. Since
$P_\sM(I-zT)^{-1}\uphar\sM=(I-z\Omega(z))^{-1}$, we get
\[
\begin{array}{ll}
 P_\sM(I-zT_a)^{-1}\uphar\sM
 &=\; \cfrac{a}{z+a}I_\sM+\cfrac{z(1-a^2)}{(z+a)(1+az)}\left(I_\sM-\cfrac{z+a}{1+az}\Omega\left(\cfrac{z+a}{1+az}\right)\,\right)^{-1}\\
 &=\; \cfrac{1}{1+az}\left(I_\sM-a\,\Omega\left(\cfrac{z+a}{1+az}\right)\right)\left(I_\sM-\cfrac{z+a}{1+az}\,\Omega\left(\cfrac{z+a}{1+az}\right)\,\right)^{-1}.
\end{array}
\]
Now consider the passive selfadjoint system
\[
\eta_a=\{T_a;\sM,\sM,\cK\},\quad T_a=(T-a I)(I-aT)^{-1},
\]
and let $\Omega_{\eta_a}$ be the transfer function of $\eta_a$. Then
from $P_\sM(I-zT_a)^{-1}\uphar\sM=(I_\sM-z \Omega_{\eta_a}(z))^{-1}$
we get
\[
(I_\sM-z\Omega_{\eta_a}(z)^{-1}=\cfrac{1}{1+az}\left(I_\sM-a\,\Omega\left(\cfrac{z+a}{1+az}\right)\right)
\left(I_\sM-\cfrac{z+a}{1+az}\,\Omega\left(\cfrac{z+a}{1+az}\right)\right)^{-1}.
\]
Hence,
\[
\Omega_{\eta_a}(z)=\left(\Omega\left(\cfrac{z+a}{1+az}\right)-a\,I_\sM\right)\left(I_\sM-a\,\Omega\left(\cfrac{z+a}{1+az}\right)\,\right)^{-1}.
\]
Since
\[
\begin{array}{ll}
 \bigcap\limits_{z\in\dD}\ker\left(P_\sM(I-zT_a)^{-1}\right)
 &=\; \bigcap\limits_{z\in\dD}\ker\left(P_\sM\left(I-\cfrac{z+a}{1+az}T\right)^{-1}(I-aT)\right)\\
 &=\; (I-aT)^{-1}\bigcap\limits_{\mu\in\dD}\ker\left(P_\sM(I-\mu T)^{-1}\right),
\end{array}
\]
we conclude that if $\tau$ is minimal then also $\eta_a$ is minimal.
\end{proof}

\begin{corollary}\label{Fruhling}
Let $\tau=\{T;\sM,\sM,\cK\}$ be a passive selfadjoint system with
transfer function $\Omega.$  Let $a\in(-1,1)$ and suppose that
$\sigma_a=\{\cT(a);\sM,\sM,\cK\} $ is a passive selfadjoint system
with transfer function $\Omega\left(\cfrac{z-a}{1-az}\right)$; see
Theorem \ref{ksia}. Then the passive selfadjoint system
\[
\zeta_a=\{(\cT(a))_a;\sM,\sM,\cK\},\; (\cT(a))_a:=(\cT(a)-aI)(I- a\cT(a))^{-1}
\]
has the transfer function
\[
\Omega_{\zeta_a}(z)=(\Omega(z)-a\,I)(I-a\,\Omega(z))^{-1},\;z\in\dC\setminus\{(-\infty,-1]\cup[1,+\infty)\}.
\]
If $\tau$ is minimal then $\zeta_a$ is minimal, too.
\end{corollary}

The next result shows that the Redheffer product ${\bf
K}_{-a}\bullet{\bf V}_a(T)$ coincides with $W_{-a}(T)$.

\begin{proposition}\label{cnhfyyj}
Let the block operator $T$ in \eqref{Tblock2} be a selfadjoint
contraction, let $\Omega(z)=D+zC(I-zF)^{-1}C^*$, and denote
\[
\wh T_a=\begin{bmatrix}\Omega(a)&\sqrt{1-a^2}C(I-aF)^{-1}\cr \sqrt{1-a^2}(I-aF)^{-1}C^*&F_a \end{bmatrix} :\begin{array}{l}\sM\\ \oplus\\\cK\end{array}\to \begin{array}{l}\sM\\ \oplus\\\cK\end{array}
\]
and
\[
{\bf K}_{-a}=\begin{bmatrix}-aI &\sqrt{1-a^2}I\cr
\sqrt{1-a^2}& a I \end{bmatrix}:\begin{array}{l}\sM\\\oplus\\ \sM\end{array}\to \begin{array}{l}\sM\\\oplus\\ \sM\end{array}.
\]
Then the Redheffer product ${\bf K}_{-a}\bullet \wh T_a$ satisfies
the equality
\begin{equation}\label{intrav1}
{\bf K}_{-a}\bullet \wh T_a=T_a\left(=(T-a I)(I-a T)^{-1}\right).
\end{equation}
\end{proposition}
\begin{proof}
It follows from \eqref{xfcnyck} that the mapping ${\bf
K}_{-a}\bullet\wh T_a:\sM\oplus\cK\to \sM\oplus\cK$ has the form
\[
{\bf K}_{-a}\bullet\wh T_a= 
{\SMALL\setlength{\arraycolsep}{2pt}
\begin{bmatrix}(aI-\Omega(a))(I-a\Omega(a))^{-1}&(1-a^2)(I-a\Omega(a))^{-1}C(I-aF)^{-1}\cr
 (1-a^2)C^*(I-aF)^{-1}(I-a\Omega(a))^{-1}&F_a+a(1-a^2)(I-aF)^{-1}C^*(I-a\Omega(a))^{-1}C(I-aF)^{-1} \end{bmatrix}}.
\]
Comparing this with \eqref{attrns} leads to \eqref{intrav1}.
\end{proof}

\begin{theorem}\label{infix}

1) If the function $\Omega$ from $\cRS(\sM)$ is inner, then the
equality
\begin{equation}\label{fixin}
 \Omega(z)=\left(\Omega\left(\cfrac{z+a}{1+az}\right)-a\,I_\sM\right)\left(I_\sM-a\,\Omega\left(\cfrac{z+a}{1+az}\right)\,\right)^{-1}
\end{equation}
holds for all $a\in(-1,1)$ and $z\in\dC\setminus\{(-\infty,-1]\cup
[1,+\infty)\}$.

2) If $\Omega\in\cRS(\sM)$ and
 \eqref{fixin} holds for some $a\in(-1,1)$, $a\ne 0$, then $\Omega$ is an inner function.
\end{theorem}
\begin{proof}
1) If $\Omega\in\cRS(\sM)$ is an inner function, then it takes the
form \eqref{forminner} and $D=\Omega(0)$. The equality \eqref{fixin}
can be verified with a straightforward calculation.

2) Suppose that \eqref{fixin} holds for some $a\in(-1,1)$. Then the
equality
\[
\Omega\left(\cfrac{z+a}{1+az}\right)-a\,I= \Omega(z)\left(I-a\,\Omega\left(\cfrac{z+a}{1+az}\right)\,\right)
\]
holds for all $ z\in\dC\setminus\{(-\infty,-1]\cup [1,+\infty)\}$.
Letting $z\to\pm 1$, we get the equalities
$\Omega(1)^2=\Omega(-1)^2=I_\sM.$ Moreover, with $z=0$ we get from
\eqref{fixin} the equality
\[
(\Omega(a)-aI_\sM)(I_\sM-a\Omega(a))^{-1}=\Omega(0).
\]
Then by applying Theorem \ref{thinne} one finally concludes that
$\Omega$ is an inner function.
\end{proof}

\subsection{The functional equation $\Omega(z)=\left(\Omega\left(\cfrac{z-a}{1-az}\right)-a\,I_\sM\right)\left(I_\sM-a\,\Omega\left(\cfrac{z-a}{1-az}\right)\,\right)^{-1}$}

\begin{theorem}\label{cjdcyjd}
Let $a\in(-1,1)$, $a\ne 0$. Then the equality
\begin{equation}\label{fixinnew}
 \Omega(z)=\left(\Omega\left(\cfrac{z-a}{1-az}\right)-a\,I_\sM\right)\left(I_\sM-a\,\Omega\left(\cfrac{z-a}{1-az}\right)\,\right)^{-1}
\end{equation}
holds for all $z\in\dC\setminus\{(-\infty,-1]\cup
[1,+\infty)\}$ and for some $\Omega\in\cRS(\sM)$ if and only if $\Omega$ is identically equal to a fundamental symmetry in $\sM$.
\end{theorem}
\begin{proof}
We will use the M\"{o}bius representation \eqref{mobinn} for $\Omega\in\cRS(\sM)$,
\begin{equation}\label{mobinn3}
\Omega(z)=\Omega(0)+D_{\Omega(0)}\Lambda(z)\left(I+\Omega(0)\Lambda(z)\right)^{-1}D_{\Omega(0)},\quad  z\in\dC\setminus\{(-\infty,-1]\cup[1,+\infty)\},
\end{equation}
with a function $\Lambda\in\cRS(\sD_{\Omega(0)})$ such that $\Lambda(z)=z\Gamma(z)$, where $\Gamma$ is a holomorphic $\bB(\sD_{\Omega(0)})$-valued function 
with $\|\Gamma(z)\|\le 1$ for $ z\in\dD$; see Proposition \ref{schaffen}.

Equality \eqref{fixinnew}
is equivalent to the equality
\[
\left(\Omega(z)-a I_\sM\right)\left(I_\sM-a\Omega(z)\right)^{-1}=\Omega\left(\cfrac{z+a}{1+z a}\right)\;\forall z\in\dC\setminus\{(-\infty,-1]\cup[1,+\infty)\}.
\]
Now, with $z=0$ this gives the equality
\[
\left(\Omega(0)-a I_\sM\right)\left(I_\sM-a\Omega(0)\right)^{-1}=\Omega(a)\Longleftrightarrow \Omega(0)-\Omega(a)=a(I_\sM-\Omega(a)\Omega(0)).
\]
Denote $\Omega(0)=D$. Assume that $\sD_D\ne \{0\}$ and represent $\Omega\in\cRS(\sM)$ in the form \eqref{mobinn3}.
Furthermore, we use that $\Lambda(z)=z\Gamma(z)$. This leads to
\[
-aD_D(\Gamma(a)(I+aD\Gamma(a))^{-1}D_D=a\left(I_\sM-\left (D+aD_D(\Gamma(a)(I+aD\Gamma(a))^{-1}D_D\right) D\right).
\]
It follows that
\[
\begin{array}{l}
  -\Gamma(a)(I+aD\Gamma(a))^{-1} =I-a\Gamma(a)(I+aD\Gamma(a))^{-1}D \\
 \quad \Longleftrightarrow\quad (I+a\Gamma(a)D)^{-1}\Gamma(a)=a\Gamma(a)D(I+a\Gamma(a)D)^{-1}-I \\
 \quad \Longleftrightarrow\quad (I+a\Gamma(a)D)^{-1}\Gamma(a)=a\Gamma(a)D(I+a\Gamma(a)D)^{-1}-I\\
 \quad \Longleftrightarrow\quad (I+a\Gamma(a)D)^{-1}\Gamma(a)=-(I+a\Gamma(a)D)^{-1}\\
 \quad \Longleftrightarrow\quad \Gamma(a)=-I.
\end{array}
\]
Since $\Gamma(z)$ belongs to the Schur class in $\sM$, we get
$$\Gamma(z)=-I,\quad  z\in\dC\setminus\{(-\infty,-1]\cup[1,+\infty)\}.$$
Hence for all $z\in\dC\setminus\{(-\infty,-1]\cup[1,+\infty)\}$,
\[
 \Omega(z)=D-zD_D(I-z D)^{-1}D_D=(D-z I)(I-zD)^{-1}.
\]
However, the function $(D-z I)(I-zD)^{-1}$ belongs to the class $\cRS(\sM)$ if and only if it is a constant function. In other words,
one must have $\sD_D=\{0\}$. This means that $\Omega(z)\equiv D$, in $\dC\setminus\{(-\infty,-1]\cup[1,+\infty)\}$, and here
$D$ is a fundamental symmetry in $\sM$ ($D=D^*=D^{-1}$).

\end{proof}
\begin{appendices}

\section{The Schur-Frobenius formula for the resolvent} \label{AppendixA}
Let
\[
\cU=\begin{bmatrix} D&C \cr B&A\end{bmatrix} :
\begin{array}{l} \sM \\\oplus\\ \sH \end{array} \to
\begin{array}{l} \sM \\\oplus\\ \sH \end{array}
\]
be a bounded block operator. Then the resolvent
$R_\cU(\lambda)=(\cU-\lambda I)^{-1}$ of $\cU$ (the Schur-Frobenius
formula) takes the following block form:
\begin{equation}
\label{Sh-Fr1}
\begin{array}{l}
R_\cU(\lambda)=
\begin{bmatrix}-V^{-1}(\lambda)&V^{-1}(\lambda)CR_A(\lambda)\cr
R_A(\lambda)BV^{-1}(\lambda)&R_A(\lambda)\left(I_\cH-BV^{-1}(\lambda)CR_A(\lambda)\right)
\end{bmatrix},
\quad \lambda\in\rho(\cU)\cap\rho(A),
\end{array}
\end{equation}
where
\begin{equation}
\label{VT}V(\lambda):=\lambda I_\sM-D+CR_A(\lambda)B,\;
\lambda\in\rho(A).
\end{equation}
In particular, $\lambda\in\rho(\cU)\cap\rho(A)\iff
V^{-1}(\lambda)\in\bL(\sM)$ and \eqref{Sh-Fr1} and
\eqref{VT} imply
\[
\left(P_\sM R_U(\lambda)\uphar\sM\right)^{-1}
=D-CR_A(\lambda)B-\lambda I_\sM.
\]

\section{Contractive $2\times 2$ block operators}\label{AppendixB} The following well-known result gives the
structure of a contractive block operator.

\begin{proposition} \cite{AG, DaKaWe, ShYa}.
\label{ParContr} The block operator $2\times 2$ matrix
\[
T=\begin{bmatrix} D&C \cr B&F\end{bmatrix} :
\begin{array}{l} \sM \\\oplus\\ \cK \end{array} \to
\begin{array}{l} \sN \\\oplus\\ \cL \end{array}.
\]
is a contraction if and only if $D\in\bB(\sM,\sN)$ is a contraction
and the entries $B$,$C$, and $F$ take the form
\[
\begin{array}{l}
B=ND_D,\quad C=D_{D^*}G,\\
 F=-ND^*G+D_{N^*}LD_{G},
\end{array}
\]
where the operators $N\in\bB(\sD_D,\cL)$, $G\in\bB(\cK,\sD_{D^*})$
and $L\in\bB(\sD_{G},\sD_{N^*})$ are contractions. Moreover, the
operators $N,\,G,$ and $L$ are uniquely determined by $T$. Furthermore, the
following equality holds for all $f \in \sM$, $h \in \cK$:
\[
\begin{split}
\left\|\begin{bmatrix}f\cr h\end{bmatrix}\right\|^2
&-\left\|\begin{bmatrix} D &D_{D^*}G \cr
ND_{D}&-ND^*G+D_{N^*}LD_{G}\end{bmatrix}
\begin{bmatrix}f\cr h\end{bmatrix}  \right\|^2 \\
&=\|D_N(D_{D}f-D^*Gh)-N^*LD_{G}h\|^2+\|D_{L}D_{G}h\|^2.
\end{split}
\]
\end{proposition}
\begin{remark}
\label{stut1} If $\sN=\sM$, $\cL=\cK$ , then $T\in\bB(\sM\oplus\cK)$
is a selfadjoint contraction if and only if $D=D^*$, $B=C^*$,
$G=N^*$, $L=L^*$.
\end{remark}

\begin{remark}\label{herbst}
Let $F$ be a selfadjoint contraction in the Hilbert space $\cK$,
then the operator given by the block operator
\[
\dJ_F=\begin{bmatrix}-F&D_F\cr D_F&F  \end{bmatrix}:\begin{array}{l}\sD_F\\\oplus\\\cK \end{array}\to \begin{array}{l}\sD_F\\\oplus\\\cK \end{array}
\]
is selfadjoint and unitary: $\dJ_F=\dJ_F=\dJ^{-1}_F$.

Let $\sM$ be a Hilbert space, let $K\in\bB(\sD_F,\sM)$ be a contraction and let
$$\begin{bmatrix}K&0\cr 0& I\end{bmatrix}:\begin{array}{l}\sD_F\\\oplus\\\cK\end{array}\to \begin{array}{l}\sM\\\oplus\\\cK\end{array}.$$
 Then for any selfadjoint contraction $Y\in\bB(\sD_{K^*})$ the block operator
\[
\begin{array}{ll}
 T&=\; \begin{bmatrix}K&0\cr 0&I\end{bmatrix}\begin{bmatrix}-F&D_F\cr D_F&F  \end{bmatrix}
 \begin{bmatrix}K^*&0\cr 0&I\end{bmatrix}+\begin{bmatrix}D_{K^*}YD_{K^*}&0\cr
 0&0\end{bmatrix}\\[5mm]
 &=\; \begin{bmatrix}-KFK^*+D_{K^*}YD_{K^*}&KD_F\cr D_FK^*&F\end{bmatrix}:\begin{array}{l}\sM\\\oplus\\\cK \end{array}\to
 \begin{array}{l}\sM\\\oplus\\\cK \end{array}
\end{array}
\]
is selfadjoint contraction. Conversely, any selfadjoint contraction
$$T= \begin{bmatrix}D&C\cr C^*&F\end{bmatrix}:\begin{array}{l}\sM\\\oplus\\\cK \end{array}\to \begin{array}{l}\sM\\\oplus\\\cK \end{array}$$
has the representation
\[
T=\begin{bmatrix}K&0\cr 0&I\end{bmatrix}\dJ_F\begin{bmatrix}K^*&0\cr 0&I\end{bmatrix}+
\begin{bmatrix}D_{K^*}YD_{K^*}&0\cr 0&0\end{bmatrix}
\]
with some contraction $K\in\bB(\sD_F,\sM)$ and some selfadjoint contraction $Y\in\bB(\sD_{K^*})$.
Moreover, $T$ is unitary if and only if $K$ is an isometry and $Y=Y^*=Y^{-1}$ in the subspace $\sD_{K^*}=\ker K^*$.
\end{remark}

\end{appendices}



\begin{thebibliography}{99}

\bibitem{Arl1991}
Yu.M.~Arlinski\u{\i}, \textit{Characteristic functions of operators
of the class $C(\alpha)$}, Izv. Vyssh. Uchebn. Zaved. Mat., 1991,
No. 2, p. 13--21 (Russian). English translation in Soviet Math. (Iz.
VUZ), 35 no. 2 (1991), 13--23.

\bibitem{Arl_arxiv_2017}
Yu.M.~Arlinski\u{\i},
\textit{Transformations of Nevanlinna operator-functions and their fixed points},
 Methods Funct. Anal. Topology, 23 no. 3 (2017), 212--230.

\bibitem{AHS2}
Yu.M.~Arlinski\u{\i}, S.~Hassi, H.S.V.~de Snoo,
\textit{$Q$-functions of quasi-selfadjoint contractions}, Operator
theory and indefinite inner product spaces, Oper. Theory Adv. Appl.,
{163} (2006), 23--54, Birkh\"auser, Basel, 2006.

\bibitem{AHS2007}
Yu.~Arlinski\u{\i}, S.~Hassi, and H.S.V.~de~Snoo,
\textit{Parametrization of contractive block operator matrices and
passive discrete-time systems}, Complex Anal. Oper. Theory, {1}
(2007), 211--233.

\bibitem{AHS3}
Yu.M.~Arlinski\u{\i}, S.~Hassi, H.S.V.~de Snoo, \textit{Passive
systems with a normal main operator and quasi-selfadjoint systems},
Complex Anal. Oper. Theory, {3} no. 1 (2009), 19--56.

\bibitem{ArlKlotz2010}
Yu.~Arlinski\u{\i} and L.~Klotz,\textit{Weyl functions of bounded
quasi-selfadjoint operators and block operator Jacobi matrices},
Acta Sci. Math. (Szeged), {76} no. 3--4 (2010),  585--626.

\bibitem{A}
 D.Z.~Arov, \textit{Passive linear stationary dynamical
systems}, Sibirsk. Math. Journ., {20} no.2 (1979), 211--228
[Russian]. English translation in Siberian Math. Journ., {20}
(1979), 149--162.

\bibitem{Arov}
D.Z.~Arov, \textit{Stable dissipative linear stationary dynamical scattering
systems}, J. Operator Theory, 1 (1979), 95--126 (Russian)

\bibitem{ArNu1}
D.Z.~Arov and M.A.~Nudel'man, \textit{A criterion for the unitary
similarity of minimal passive systems of scattering with a given
transfer function}, Ukrain. Math. J., 52 (2000), 161--172 (Russian).


\bibitem{ArNu2}
D.Z.~Arov and M.A.~Nudel'man, \textit{Tests for the similarity of all
minimal passive realizations of a fixed transfer function
(scattering and resistance matrix)}, Mat. Sb., 193 no. 6 (2002),
3--24.


\bibitem{AG}
Gr.~Arsene and A.~Gheondea, \textit{Completing matrix contractions}, J.
Operator Theory, 7 (1982), 179-189.

\bibitem{BL1976}
J.A.~Ball and A.~Lubin, \textit{On a class of contractive
perturbations of restricted shifts}, Pacific J. Math., 63 no. 2
(1976), 309--323.


\bibitem{Ber}
Yu.M.~Berezansky, \textit{Expansion in eigenfunctions of selfadjoint operators}, Amer. Math. Soc.,
Providence, R.I., 1968.

\bibitem{Cons}
T. Constantinescu, \textit{Operator Schur algorithm and associated
functions}, Math. Balkanica (N.S.) 2  no. 2-3 (1988), 244--252.


\bibitem{DaKaWe}
Ch.~Davis, W.M.~Kahan, and H.F.~Weinberger, \textit{Norm preserving
dilations and their applications to optimal error bounds}, SIAM J.
Numer. Anal., 19 (1982), 445--469.


\bibitem{FilTsek1987}
A.P.~Filimonov and E.R.~Tsekanovski\u{\i},
\textit{Automorphic-invariant operator colligations and the
factorization of their characteristic operator-functions}. Funkts.
Anal. Prilozh., 21 No.4 (1987), 94--95 [Russian]. English
translation in Funct. Anal. Appl., 21 No.4 (1987),  343--344.

\bibitem{Redheffer}
R.M.~Redheffer, \textit{On certain linear fractional
transformation}, J.Math. Phys., 39 (1960), 269--286.


\bibitem{Shmul1}
Yu.L.~Shmul'yan, \textit{Generalized fractional-linear
transformations of operator balls}, Sibirsk. Mat. Zh. 21 (1980),
No.5, 114--131 [Russian]. English translation in Siberian
Mathematical Jour., 21 No.5 (1980), 728--740.

\bibitem{ShYa}
Yu.L.~Shmul'yan and R.N.~Yanovskaya, \textit{Blocks of a contractive
operator matrix}, Izv. Vuzov, Mat., 7 (1981), 72-75. [Russian]


\bibitem{SF}
B.~Sz.-Nagy and C.~Foias, \textit{Harmonic analysis of operators on
Hilbert space}, North-Holland, New York, 1970.

\bibitem{Timotin1995}
D.~Timotin, \textit{Redheffer products and characteristic
functions},  J. Math. Anal. Appl., 196 no. 3 (1995), 823--840.

\end{thebibliography}
\end{document}